\documentclass[11pt,a4paper]{article}

\usepackage[left=2.3cm, top=2.3cm,bottom=2.3cm,right=2.3cm]{geometry}

\usepackage{mathtools,amssymb,amsthm,mathrsfs,calc,graphicx,stmaryrd,xcolor,dsfont,bbm,float,array,epsfig,hyperref,color}
\usepackage[numbers,square]{natbib}
\usepackage[british]{babel}
\usepackage{amsfonts,amsmath,amsthm,amssymb}
\usepackage{verbatim}
\usepackage{multirow}
\usepackage{tikz}

\numberwithin{equation}{section}
\numberwithin{figure}{section}

\newtheorem{thm}{Theorem}[section]
\newtheorem{lem}[thm]{Lemma}
\newtheorem{cor}[thm]{Corollary}
\newtheorem{prop}[thm]{Proposition}

\theoremstyle{remark}
\newtheorem{rem}[thm]{Remark}

\theoremstyle{definition}

\numberwithin{equation}{section}

\def\P{{\mathbb{P}}}
\def\Z{{\mathbb{Z}}}
\def\mz{{\mathbb{Z}}}
\def\R{{\mathbb{R}}}
\def\V{{\mathbb{V}}}
\def\F{{\mathcal{F}}}

\newcommand{\od}{\overset{{\rm d}}{=}}
\newcommand{\dod}{\overset{\rm{d}}{\longrightarrow}}
\newcommand{\topr}{\overset{\mathbb{P}}{\longrightarrow}}

\newcommand{\PP}{\mathbb{P}}
\newcommand{\E}{\mathbb{E}}
\newcommand{\me}{\mathbb{E}}
\newcommand{\Y}{\mathbb{Y}}
\newcommand{\G}{\mathcal{G}}
\newcommand{\X}{\mathcal{X}}

\renewcommand{\P}{\mathbb{P}}
\newcommand{\NN}{\mathbb{N}}
\newcommand{\N}{\mathbb{N}}

\newcommand{\ZZ}{\mathbb{Z}}
\newcommand{\W}{\mathbb{W}}
\newcommand{\Zc}{Z^{\rm crit}}
\newcommand{\Wc}{W^{\rm crit}}
\newcommand{\Yc}{Y^{\rm crit}}

\newcommand{\s}{\sigma}
\newcommand{\8}{\infty}
\newcommand{\wt}{\widetilde}
\renewcommand{\a}{\alpha}

\renewcommand{\l}{\lambda}

\newcommand{\ma}{\color{magenta}}

\begin{document}

\title{\bfseries  Stable limit laws for random walk in a sparse random environment I: moderate sparsity}

\author{Dariusz Buraczewski, Piotr Dyszewski, Alexander Iksanov,\\ Alexander Marynych and Alexander Roitershtein}

\date{}

\maketitle

\begin{abstract}
A random walk in a sparse random environment is a model introduced
by Matzavinos~et~al. [Electron. J. Probab. 21, paper no.~72: 2016]
as a generalization of both a simple symmetric random walk and a
classical random walk in a random environment. A random walk
$(X_n)_{n\in \N\cup\{0\}}$ in a sparse random environment
$(S_k,\lambda_k)_{k\in\mz}$ is a nearest neighbor random walk on
$\mz$ that jumps to the left or to the right with probability
$1/2$ from every point of $\mz\setminus
\{\ldots,S_{-1},S_0=0,S_1,\ldots\}$ and jumps to the right (left)
with the random probability $\lambda_{k+1}$
($1-\lambda_{k+1}$) from the point $S_k$, $k\in\mz$. Assuming
that $(S_k-S_{k-1},\lambda_k)_{k\in\mz}$ are independent copies of
a random vector $(\xi,\lambda)\in \mathbb{N}\times (0,1)$ and the
mean $\me \xi$ is finite (moderate sparsity) we obtain stable
limit laws for $X_n$, properly normalized and centered, as
$n\to\infty$. While the case $\xi\leq M$ a.s.\ for some
deterministic $M>0$ (weak sparsity) was analyzed by
Matzavinos~et~al., the case $\me \xi=\infty$ (strong sparsity)
will be analyzed in a forthcoming paper. \noindent
\end{abstract}

\tableofcontents

\section{Introduction}
Simple random walks on $\mathbb{Z}$ (the set of integers) arise in
various areas of classical and modern stochastics. However, their
intrinsic homogeneity reduces in some situations applicability of
the simple random walks. Solomon~\cite{solomon1975random}
eliminated this drawback by introducing a random environment which
made a modified random walk space inhomogeneous. In the present
article we investigate an intermediate model, called random walk
in a sparse random environment (RWSRE), in which homogeneity of an
environment is only perturbed on a sparse subset of $\mathbb{Z}$.
Since RWSRE is a particular case of a random walk in a random
environment (RWRE) we proceed by recalling the definition of the
latter.

Set $\Omega = (0,1)^\Z$ and $\mathcal{X} = \Z^\N$. Let $\F$ be the
Borel $\sigma$-algebra of subsets of $\Omega$, $P$ a probability
measure on $(\Omega, \F)$ and $\G$ the $\sigma$-algebra generated
by the cylinder sets in $\mathcal{X}$. A random environment is a
random element $\omega = (\omega_n)_{n\in \Z}$ of the measurable
space $(\Omega,\F)$ distributed according to $P$. A quenched
(fixed) environment $\omega$ provides us with a probability
measure $\P_\omega$ on $\X$ whose transition kernel is given by
$$\P_\omega \{X_{n+1}=j| X_n = i\} = \left\{
\begin{array}{cl}
\omega_i, & \mbox{if } j=i+1,\\
1- \omega_i,\ & \mbox{if } j=i-1,\\
0, & \mbox{otherwise.}
\end{array}\right.$$
With the initial condition $X_0:=0$ the sequence $X=(X_n)_{n\in
\N_0}$ is a Markov chain on $\Z$ (under $\P_\omega$) which is
called random walk in the random environment $\omega$. Here and
hereafter, $\N_0:=\N\cup\{0\}$. It is natural to investigate RWRE
from two viewpoints which are different in many aspects: under the
quenched measure $\P_\omega$ for almost all (with respect to $P$)
$\omega$, that is, for a typical $\omega$ or under an annealed
measure. Formally, the annealed measure $\P$ on $(\Omega\times \X,
\F\otimes \G)$ is defined as a semi-direct product $\P = P \ltimes
\P_{\omega}$ via the formula $$\P\{F\times G\} = \int_F
\P_{\omega}\{G\} P({\rm d}\omega),\quad F\in \F,\quad G\in \G.$$
Note that in general $X$ is no longer a Markov chain under $\P$.
Usually one assumes that an environment $\omega$ forms a
stationary and ergodic sequence or even a sequence of iid
(independent and identically distributed) random variables. In
this setting RWRE has attracted a fair amount of attention among
probabilistic community resulting in quenched and annealed limit
theorems~\cite{bouchet2016quenched,dolgopyat2012quenched,enriquez2009limit,kesten1986limit,kesten1975limit,sinai1982limit,sznitman1999law}
and large deviations~\cite{buraczewski2017precise,
comets2000quenched,dembo1996tail,
gantert1998quenched,greven1994large,pisztora1999large,pisztora1999precise,varadhan2003large,zerner1998lyapounov}.
This list of references is far from being complete.

We aim at establishing annealed limit theorems for $X$ (that is,
under $\P$) in a so called sparse random environment which
corresponds to a particular choice of $P$ which is specified as
follows.  Let $((\xi_k,\lambda_k))_{k\in \Z}$ be a sequence of
independent copies of a random vector $(\xi,\lambda)$ which
satisfies $\lambda \in (0,1)$ and $\xi \in \N$ a.s.  For $n\in
\Z$, set
$$S_n = \left\{ \begin{array}{cl}
\sum_{k=1}^n \xi_k, & \mbox{if } n>0, \\
0, & \mbox{if } n=0, \\
-\sum_{k=n+1}^{ 0} \xi_k, & \mbox{if } n<0.
\end{array} \right.
$$
The sparse random environment $\omega = (\omega_n)_{n\in\Z}$ is defined by
\begin{equation}\label{eq:sparse}
\omega_n = \left\{
\begin{array}{ll}
\l_{k+1}, \ & \mbox{if $n=S_k$ for some }k\in\Z,\\
\frac{1}{2}, \ & \mbox{otherwise.}
\end{array}\right.
\end{equation}
The model (with $\lambda_k$ in \eqref{eq:sparse} replacing
$\lambda_{k+1}$) was introduced by Matzavinos, Roitershtein and
Seol~\cite{matzavinos2016random}. These authors obtained various
results including a recurrence/transience criterion, a strong law
of large numbers and limit theorems. However, many results in
\cite{matzavinos2016random} were proved under quite restrictive
conditions including boundedness of $\xi$, a strong ellipticity
condition for the distribution of $\lambda$ and independence of
$\xi$ and $\lambda$. In this setting some essential properties of
$X$ remain hidden. Our main purpose is to relax the aforementioned
assumptions substantially, thereby establishing limit theorems in
full generality, and to find out how distributional properties of
the vector $(\xi,\lambda)$ affect the asymptotic behavior of $X$.
It turns out that the asymptotics of $X$ is regulated by the tail
behaviors of $\xi$ and $\rho:=(1-\lambda)/\lambda$ which determine
sparsity of the environment and the local drift of the
environment, respectively. In this paper we investigate the case
where $\me\xi<\infty$. We call the corresponding environment
`moderately sparse', whereas in the opposite case where
$\me\xi=\infty$ we say that the environment is `strongly sparse'.
The analysis of $X$ in a strongly sparse environment requires
completely different techniques and will be carried out in a
companion paper \cite{BurDysIksMarRoi:2018+}.

The present article is organized as follows. In Section
\ref{sec:results} we formulate our limit theorems for $X$ and the
first passage times of $X$. In Section \ref{sec:bpre1} we describe
our approach and define a branching process $Z$ in a random
environment which is used to analyze the random walk $X$. In
Section \ref{subsec:notation} we introduce necessary notation
related to the process $Z$. In Section \ref{sec:strategy} we
explain a heuristic behind our proof and present a number of
important estimates and decompositions used throughout the paper.
Among other things, we demonstrate in this section how to reduce
the initial problem to the asymptotic analysis of sums of certain
iid random variables. The tail behavior of these variables is
discussed in Section \ref{sec:tails}. Section \ref{sec:GW} is
devoted to the analysis of a particular critical Galton--Watson
process with immigration which naturally arises in the context of
random walks in the sparse random environment. The proofs of the
main results are given in Sections \ref{sec:lln}, \ref{7.2} and
\ref{7.3}. The proofs of auxiliary lemmas can be found in Section
\ref{sec:proofs_lemmas} and the Appendix.

\section{Main results}\label{sec:results}

We focus on the case when $X$ is $\P$-a.s.~transient to $+\infty$
and the environment is moderately sparse, that is, $\me
\xi<\infty$. Recall the notation $$\rho = \frac{1-\l}{\l}.$$
According to Theorem 3.1 in \cite{matzavinos2016random}, $X$ is
$\P$-a.s.~transient to $+\infty$ if
\begin{equation}\label{eq:right_transience}
\E \log \rho\in [-\infty, 0) \quad \text{and}\quad \E\log
\xi<\infty.
\end{equation}
The first inequality excludes the degenerate case $\rho=1$ a.s.\
in which $X$ becomes a simple random walk. The second inequality
is always true for the moderately sparse environment. We note
right away that our standing assumptions $\E \log \rho\in
[-\infty, 0)$ and $\E\xi<\infty$ hold under the conditions of our
main results, Theorems \ref{thm:main1T} and \ref{thm:main2T}.

The sequence $(T_n)_{n\in\mz}$ of the first passage times defined by
$$ T_n = \inf\{k\geq 0: X_k = n \},\quad n\in\mz$$ is of crucial importance for our arguments. Of course, the observation that the asymptotics of
$X$ can be derived from that of $(T_n)$ is not new and has been
exploited in many earlier papers in the area of random walks in
random environments. Assuming only transience to the right it is
shown on p.~12 in \cite{matzavinos2016random} that
$$\lim_{n\to\infty}\frac{T_{S_n}}{n}=\me T_{S_1}\quad
\P-\text{a.s.}$$ This in combination with Lemma 4.4 in
\cite{matzavinos2016random} leads to the conclusion that
\begin{equation}\label{def_v}
\lim_{n \to \infty} \frac{X_n}{n} =\me\xi/ \me T_{S_1}=: v\quad
\text{and}\quad \lim_{n \to \infty} \frac{T_n}{n}
=\frac{1}{v}\quad \P-\text{a.s.}
\end{equation}
whenever the environment is moderately sparse. Furthermore, under
the additional assumption that $\xi$ and $\lambda$ are
independent, Theorem~3.3 in~\cite{matzavinos2016random} states
that
\begin{equation}\label{eq:v_in_roiterstein_et_al}
v = \frac{(1-\E\rho)\E \xi }{ (1-\E\rho)\E \xi^2 + 2\E\rho (\E \xi)^2}
\end{equation}
provided that $\E \rho <1$ and $\E \xi^2<\infty$, and $v=0$, otherwise.

In Proposition \ref{thm:lln} we give an explicit formula for $v$
when $\xi$ and $\lambda$ are allowed to be dependent.
\begin{prop}\label{thm:lln}
Assume that $\E \log \rho\in [-\infty, 0)$ and $\E \xi<  \infty$.
Then
\begin{equation}\label{formv}
v =\frac{(1-\E \rho)\E \xi}{ (1-\E \rho)\E \xi^2 + 2 \E \xi \E
\rho\xi}, \qquad\frac{1}{v} = \frac{1}{\E \xi} \left( \E \xi^2 +
\frac{2 \E \xi \E\rho \xi}{1- \E \rho} \right)
\end{equation}
provided that $\E\rho<1$, $\E \rho\xi<\infty$ and $\E
\xi^2<\infty$, and $v=0$ ($1/v=\infty$), otherwise.
\end{prop}

Turning to weak convergence results we first formulate our
assumptions on the distribution of $\rho$. Two different sets of
conditions will be used:
\begin{description}\itemsep3mm
\item[($\rho 1$)] for some $\alpha\in (0,2]$
\begin{equation*}
\E \rho^{\alpha}=1,\quad \E \rho^{\alpha}\log^{+}\rho<\infty\quad
\text{and}\quad\text{ the distribution of }\log \rho\text{ is
nonarithmetic},
\end{equation*}
where $\log^{+}x:=\max( 0,\log x)$;
\item[($\rho 2$)] there exists an open interval $\mathcal{I}\subset (0,\infty)$ such that $\E \rho^x<1$ for all $x\in \mathcal{I}$.
\end{description}
Assuming that $(\rho 1)$ holds for some $\alpha>0$ we further
distinguish two cases pertaining to the distribution of $\xi$:
\begin{description}\itemsep3mm
\item[($\xi 1$)]  $\E \xi^{2\alpha\vee 1}<\infty$, where $x\vee y:=\max(x,y)$;

\item[($\xi 2$)] there exists a slowly varying function $\ell$ such that
\begin{equation}\label{eq:reg_var_xi}
\P\{\xi>t\}\sim t^{-\beta}\ell(t),\quad t\to\infty
\end{equation}
for some $\beta\in  (1,2\alpha]$, and $\E
\xi^{2\alpha}=\infty$ if $\beta=2\alpha$.
\end{description}
Finally, if $(\rho 2)$ holds for some open interval $\mathcal{I}$
we assume that either ($\xi 1$) holds for some $\alpha\in
\mathcal{I}$ or the regular variation assumption in ($\xi 2$)
holds for some $\beta$ satisfying $\beta/2\in \mathcal{I}$.

We summarize our results in Table \ref{table1} with an emphasis on
which component of the environment dominates\footnote{In some
cases we also need additional technical assumptions concerning the
joint distribution of $\rho$ and $\xi$, for instance, $\E (\rho
\xi)^\alpha <\infty$. These will be stated explicitly in the
corresponding theorems.}.

\begin{table}[!hb]
\centering
\begin{tabular}{|p{1cm}|p{5cm}|p{8cm}|}
\hline
& ($\xi 1$) & ($\xi 2$)\\
\hline
\multirow{5}*{$(\rho 1)$} & & If $\beta<2\alpha$, see ($\rho 2$) with $\alpha=\beta/2$\\
\cline{3-3}
& & If $\beta=2\alpha,\;\lim_{t\to\infty}\ell(t)=0$, then $\rho$ dominates (Thm. \ref{thm:main1T} {\rm (A2)})\\
\cline{3-3}
& $\rho$ dominates (Thm. \ref{thm:main1T} {\rm (A1)}) & If $\beta=2\alpha,\;\lim_{t\to\infty}\ell(t)=C_{\ell}\in(0,\infty)$, then contributions of $\rho$ and
$\xi$ are comparable (Thm. \ref{thm:main1T} {\rm (A3)})\\
\cline{3-3}
&  & If $\beta=2\alpha,\;\lim_{t\to\infty}\ell(t)=+\infty$, then $\xi$ dominates (Thm. \ref{thm:main2T} {\rm (B1)})\\
\cline{3-3}
& & If $\beta>2\alpha \Longrightarrow \E \xi^{2\alpha}<\infty$, see ($\rho 1$) and ($\xi 1$)\\
\hline $(\rho 2)$ & $2\in \mathcal{I}$, contributions of $\rho$
and $\xi$ are comparable (Prop. \ref{thm:mainD}) &  $\beta\in
(1,4)$ and
$\beta/2\in \mathcal{I}$ $\Longrightarrow \xi$ dominates (Thm. \ref{thm:main2T} {\rm (B2)})\\
\hline
\end{tabular}
\caption{Influence of the environment and limit theorems for
$T_n$.} \label{table1}
\end{table}

In what follows, for $\alpha\in (0,2)$, we denote by
$\mathcal{S}_\alpha$ a random variable with an $\alpha$-stable
distribution defined by $$-\log \E\exp
(-u\mathcal{S}_\alpha)=\Gamma(1-\alpha)u^\alpha,\quad u\geq 0,$$
where $\Gamma(\cdot)$ is the gamma function, if $\alpha\in (0,1)$;
$$\log \E \exp ({\rm i}u \mathcal{S}_1)=-(\pi/2)|u|-{\rm i}u\log
|u|,\quad u\in\R;$$ $$\log \E \exp({\rm
i}u\mathcal{S}_\alpha)=|u|^\alpha
\frac{\Gamma(2-\alpha)}{\alpha-1}(\cos (\pi\alpha/2)-{\rm i}\sin
(\pi\alpha/2){\rm sign}\,u),\quad u\in\R,$$ if $\alpha\in (1,2)$.
Note that $\mathcal{S}_\alpha$ is a positive random variable when
$\alpha\in (0,1)$ and it has a spectrally positive $\alpha$-stable
distribution when $\alpha\in [1,2)$. Throughout the paper $\dod$
and $\topr$ will mean convergence in probability and convergence
in distribution, respectively.

In Theorem \ref{thm:main1T} and Corollary \ref{cor:main1X} we
treat the case $(\rho 1)$.
\begin{thm}\label{thm:main1T}
Assume that one of the following sets of assumptions is satisfied:
\begin{itemize}
\item[{\rm (A1)}] $(\rho 1)$ holds for some
 $\alpha\in (0, 2]$, $(\xi 1)$ holds and
$\E(\rho\xi)^\alpha<\infty$;
\item[{\rm (A2)}] $(\rho 1)$ holds for some
$\alpha\in (1/2,2]$ and $(\xi 2)$ holds with $\beta=2\alpha$ and
$\lim_{t\to\infty}\ell(t)=0$, and
$\E(\rho\xi)^\alpha <\infty$;
\item[{\rm (A3)}] $(\rho 1)$ holds for some
$\alpha\in (1/2,2)$, $(\xi 2)$ holds with $\beta=2\alpha$ and
$\lim_{t\to\infty}\ell(t)=C_{\ell}\in(0,\infty)$,
$\E\rho^{\alpha+\varepsilon}<\infty$ and $\E \rho^\alpha
\xi^{\alpha+\varepsilon}<\infty$ for some $\varepsilon>0$.
\end{itemize}
Then there exist absolute constants $A_\alpha$, $B_\alpha$ and
$C_1$ such that the following limit relations hold as
$n\to\infty$.
\begin{itemize}
\item If $\alpha\in (0, 1)$, then $\frac{T_n}{B_\alpha n^{1/\alpha}}~\dod~ \mathcal{S}_{\alpha}$.

\item If $\alpha=1$, then $\frac{T_n-A_1 a(n)}{B_1 n}~\dod~C_1+\mathcal{S}_1$,
where $a(n)\sim n\log n$.
\item If $\alpha\in(1,2)$, then $\frac{T_n-A_\alpha n}{B_\alpha n^{1/\alpha}}~\dod~ \mathcal{S}_{\alpha}$.

\item If $\alpha=2$, then $\frac{T_n-A_2 n}{B_2 (n\log n)^{1/2}}~\dod~ \mathcal{N}(0,1)$, where $\mathcal{N}(0,1)$ is a standard normal random variable.
\end{itemize}
\end{thm}
\begin{rem}\label{remark1}
See \eqref{alp=1}, \eqref{alp>1} and \eqref{alp<1} for explicit
forms of the constants $A_\alpha$, $B_\alpha$ and $C_1$. In Theorem
\ref{thm:main1T} we do not specify the constants by two reasons.
First, these involve characteristics of random variables that have
not been introduced so far. Second, some of these constants are
essentially implicit in the sense that these cannot be calculated.
\end{rem}

From Theorem \ref{thm:main1T} we deduce the following corollary.
\begin{cor}\label{cor:main1X}
Under the assumptions and notation of Theorem \ref{thm:main1T} the
following limit relations hold as $k\to\infty$.
\begin{itemize}
\item If  $\alpha\in (0, 1)$, then $\frac{X_k}{B^{-\alpha}_\alpha k^\alpha}~\dod~ \mathcal{S}_\alpha^{-\alpha}$.

\item If $\alpha=1$, then $\frac{X_k-A_1^{-1}\widehat{a}(k)}{A_1^{-2}B_1 k(\log k)^{-2}}~ \dod~ -C_1-\mathcal{S}_1$, where
$\widehat{a}(k)\sim k(\log k)^{-1}$.
\item If $\alpha\in(1,2)$, then $\frac{X_k-A^{-1}_{\alpha}k}{A_\alpha^{-(1+1/\alpha)}B_\alpha k^{1/\alpha}}~\dod~ -\mathcal{S}_\alpha$.
\item If $\alpha=2$, then $\frac{X_k-A^{-1}_2 k}{A_2^{-3/2}B_2 (k\log k)^{1/2}}~\dod~ \mathcal{N}(0,1)$.
\end{itemize}
\end{cor}
\begin{rem}
When $\alpha\in (0,1)$ the distribution of
$\mathcal{S}_\alpha^{-\alpha}$ is called the Mittag-Leffler
distribution with parameter $\alpha$. The term stems from the
facts that
$$\E\exp(u\Gamma(1-\alpha)\mathcal{S}_\alpha^{-\alpha})=\sum_{n\geq 0}\frac{u^n}{\Gamma(1+n\alpha)},\quad
u\in\R$$ and that the right-hand side defines the Mittag-Leffler
function with parameter $\alpha$.
\end{rem}

Our next theorem treats weak convergence of $T_n$ in cases where
$\xi$ plays a dominant role.
\begin{thm}\label{thm:main2T}
Assume that one of the following sets of assumptions is satisfied:
\begin{itemize}
\item[{\rm (B1)}] $(\rho 1)$ holds for some $\alpha\in (1/2,2]$, $(\xi 2)$ holds with $\beta=2\alpha$ and $\lim_{t\to\infty}\ell(t)=+\infty$, and $\E(\rho\xi)^\alpha<\infty$;
\item[{\rm (B2)}] $(\rho 2)$ holds and $(\xi 2)$ holds with $\beta\in(1,4)$ such that $\beta/2\in\mathcal{I}$ and $\E(\rho\xi)^{\beta/2+\varepsilon}<\infty$ for some $\varepsilon>0$.
\end{itemize}
In the case ${\rm (B2)}$ put $\alpha:=\beta/2$. Then there exist the functions $c_\alpha(t)$ for $\alpha\in (1/2, 2)$,
$q_1(t)$ and $r_2(t)$ regularly varying at $\infty$ of indices $1/\alpha$, $1$ and $1/2$, respectively, and the absolute constants $A_\alpha^\ast$
and $B_\alpha^\ast$ for $\alpha\in (1/2, 2]$ such that the following limit relations hold as $n\to\infty$.
\begin{itemize}
\item If $\alpha\in(1/2,1)$, then $\frac{T_n}{B^\ast_\alpha c_\alpha(n)}~\dod~ \mathcal{S}_\alpha$.
\item If $\alpha=1$, then $\frac{T_n-n-q_1(A^\ast_1 n)}{B^\ast_1 c_1(n)}~\dod~ \mathcal{S}_{1}$.

\item If $\alpha\in(1,2)$, then $\frac{T_n-A^\ast_\alpha n}{B^\ast_\alpha c_\alpha(n)}~\dod~ \mathcal{S}_\alpha$.
\item If $\alpha=2$, then $\frac{T_n-A^\ast_2 n}{B^\ast_2 r_2(n)}~\dod~ \mathcal{N}(0,1)$.
\end{itemize}
\end{thm}
\begin{rem}
This is a counterpart of Remark \ref{remark1}. Explicit forms of
the normalizing and centering sequences in Theorem
\ref{thm:main2T} and Corollary \ref{cor:main2X} given below can be
found in \eqref{IIalp<1}, \eqref{IIalp=1}, \eqref{IIalp>1} and
\eqref{IIalp=2}, and \eqref{IIIalp<1}, \eqref{IIIalp=1},
\eqref{IIIalp>1} and \eqref{IIIalp=2}, respectively.
\end{rem}

Before formulating the corresponding limit theorems for $X_k$ we
need to introduce more notation. For $\alpha\in (1/2, 1)$, denote
by $c_\alpha^\leftarrow(t)$ any positive function satisfying
$c_\alpha(c_\alpha^\leftarrow(t))\sim c_\alpha^\leftarrow
(c_\alpha(t)) \sim t$ as $t\to\infty$. Since $c_\alpha(t)$ is
regularly varying at $\infty$ such $c_\alpha^\leftarrow (t)$ do
exist by Theorem 1.5.12 in \cite{bingham1989regular}.
\begin{cor}\label{cor:main2X}
Under the assumptions and notation of Theorem \ref{thm:main2T} the following limit relations hold as $k\to\infty$.
\begin{itemize}
\item If $\alpha\in(1/2,1)$, then $\frac{X_k}{(B^\ast_\alpha)^{-\alpha}c_\alpha^\leftarrow(k)}~\dod~ \mathcal{S}_\alpha^{-\alpha}$.
\item If $\alpha=1$, then $\frac{X_k-s(k)}{t(k)}~\dod~
-\mathcal{S}_1$ for appropriate sequences $s(k)$ and $t(k)$ which
are specified in formula \eqref{IIIalp=1}.
\item If $\alpha\in(1,2)$, then $\frac{X_k-(A^\ast_\alpha)^{-1}k}{(A^\ast_\alpha)^{-(1+1/\alpha)}B^\ast_\alpha c_\alpha(k)}~\dod~ -\mathcal{S}_\alpha$.
\item If $\alpha=2$, then $\frac{X_k-(A^\ast_2)^{-1}k}{(A^\ast_2)^{-3/2}B^\ast_2
r_2(k)}~\dod~ \mathcal{N}(0,1)$.
\end{itemize}
\end{cor}

The last result of this section is given for completeness only. It
can be derived from a general central limit theorem (Theorem 2.2.1
in \cite{zeitouni2004}) for random walk in a stationary and
ergodic random environment. Since the sparse random environment is
not stationary in general, to apply this theorem one has to pass
to a stationary and ergodic environment. In Theorem 2.1 in
\cite{matzavinos2016random} it is shown that such a passage is
possible whenever $\E\xi<\infty$.
\begin{prop}\label{thm:mainD}
Assume that $(\rho 2)$ and $(\xi 1)$ hold for some $\alpha\geq 2$.
Then there exists $\sigma_0\in (0,\infty)$ such that, as
$n\to\infty$,
$$\frac{T_n-v^{-1}n}{\sigma_0 n^{1/2}}~\dod~\mathcal{N}(0,1)$$ and
$$\frac{X_n-vn}{\sigma_0 v^{3/2} n^{1/2}}~\dod~\mathcal{N}(0,1),$$
where $v$ is given in \eqref{formv}.
\end{prop}

\section{Branching processes in random environment with immigration}
\label{sec:bpre}

The connection between a random walk and a branching process with
immigration dates back to Harris~\cite{harris:1952}. In the
context of a random walk in a random environment this connection
was successfully used by Kozlov~\cite{kozlov:1974} and Kesten,
Kozlov and Spitzer~\cite{kesten1975limit}. In particular, these
authors have shown that the asymptotic behavior of RWRE can be
obtained from that of the total progeny of the aforementioned
branching process. Since we are going to exploit the same idea we
first recall a construction of the latter process. Most of the
material in Section \ref{sec:bpre1} can be found in
\cite{kesten1975limit}.

\subsection{Branching process with immigration}\label{sec:bpre1}

Throughout the paper the fact that $X_n \to \infty$ $\P$-a.s.\
plays a crucial role. Let $U_i^{(n)}$ be the number of steps of
the process $X$ from $i$ to $i-1$ during the time interval
$[0,T_n)$, that is, $$U_i^{(n)} = \# \big\{ k < T_n: X_k=i,
X_{k+1} = i-1 \big\}, \qquad i\le n.$$ Since $X_{T_n}=n$ and
$X_0=0$ we have, for $n\in \N$,
\begin{eqnarray*}
T_n &=& \# \mbox{ of steps during $[0,T_n)$}\\
&=& \# \mbox{ of steps to the right during $[0,T_n)$}+ \# \mbox{ of steps to the left during $[0,T_n)$}\\
&=& n + 2 \cdot \# \mbox{ of steps to the left during $[0,T_n)$}\\
&=& n + 2\sum_{i=-\infty}^n U_i^{(n)}.
\end{eqnarray*}
Recalling that the random walk $X$ is transient to the right we
infer
\begin{equation}\label{eq:negative_half_line_negligible}
\sum_{i < 0}U_i^{(n)}\leq \text{ total time spent by } X
\text{ in }(-\infty,0) < \infty\quad \text{a.s.}
\end{equation}
In particular, for any $\gamma>0$,
\begin{equation*}
n^{-\gamma}\sum_{i<0}U_i^{(n)}~\topr~ 0,\quad n\to\infty.
\end{equation*}
Thus, the asymptotics of $T_n$ as $n\to\infty$ is regulated by
that of $n+2\sum_{i=0}^n U_i^{(n)}$.

In what follows, we write ${\rm Geom}(p)$ for a geometric
distribution with success probability $p$, that is, $${\rm
Geom}(p)\{\ell\}=p(1-p)^\ell,\quad \ell\in\N_0.$$

\noindent {\sc Claim}. Let $\omega$ and $n$ be fixed. Then, for
$0\leq j\leq n$, $U^{(n)}_{n-j}$ is equal to the size of the $j$th
generation (excluding the immigrant) of an inhomogeneous branching
process with one immigrant in each generation. Under $\P_\omega$,
the offspring distribution of the immigrant and the other
particles in the $(j-1)$st generation is ${\rm
Geom}(\omega_{n-j})$.

\noindent {\sc Proof of the claim}. First note that $U_n^{(n)} =0$
because $X$ cannot reach $n$ before time $T_n$. Further,
$U_{n-1}^{(n)}= V^{(n-1)}_0$, where $V_0^{(n-1)}$ is the number of
excursions to the left of $n-1$ made by $X$ before time $T_n$.
Transitivity of $X$ entails that the $\P_\omega$-distribution of
$V_0^{(n-1)}$ is ${\rm Geom}(\omega_{n-1})$. Finally, for $2\leq j
\leq n-1$, we have
$$U^{(n)}_{n-j}= \sum_{k=1}^{U_{n-j+1}^{(n)}} V^{(n-j)}_k + V^{(n-j)}_0\quad \text{a.s.},$$
where $V^{(n-j)}_0$ denotes the number of excursions to the left
from $n-j$ before the first excursion to the left from $n-j+1$
(that is, before the time $T_{n-j+1}$) and $V^{(n-j)}_k$ denotes
the number of excursions to the left from $n-j$ during the $k$th
excursion to the left from $n-j+1$. Under $\P_\omega$, the random
variables $(V^{(n-j)}_k)_{k\geq 0}$ are iid with distribution
${\rm Geom}(\omega_{n-j})$ and also independent of
$U^{(n)}_{n-j+1}$. The proof of the claim is complete.

Reversing the order of indices leads to a branching process
$Z=(Z_k)_{k\geq 0}$ in a random environment (BPRE) with one
immigrant entering the system in each generation. From the very
beginning we stress that immigrants in our model are `artificial',
that is, even though they reproduce, they do not belong to any
generation and, as such, they are not counted. The evolution of
$Z$ can be described as follows. An immigrant enters the $0$th
generation which is originally empty, that is, $Z_0=0$. She gives
birth to a random number of offspring with
$\P_\omega$-distribution ${\rm Geom}(\omega_1)$ which form the
first generation. For $n\in\N$, an immigrant enters the $n$th
generation. She and the particles of the $n$th generation,
independently of each other and the particles in the previous
generations, give birth to random numbers of offspring with
$\P_\omega$ distribution ${\rm Geom}(\omega_{n+1})$. The number of
these newborn particles which form the $(n+1)$st generation is
given by
$$Z_{n+1}= \sum_{k=0}^{Z_n}G_k^{(n)},\quad n\in\N_0,$$ where $G_0^{(n)}$ is the number of offspring of the $(n+1)$st immigrant and, for $k\in\N$,
$G_k^{(n)}$ is the number of offspring of the $k$th particle in the $n$th generation (we set $G_k^{(n)}=0$ if the $k$th particle
in the $n$th generation does not exist). Observe that, under
$\P_\omega$, for each $n\in \N_0$, the random variables $(G_k^{(n)})_{k\ge 0}$ are iid
with distribution ${\rm Geom}(\omega_n)$ and also independent of
$Z_n$.

Note that when the random environment is sparse
(see~\eqref{eq:sparse}) and fixed, for the most time, the
branching process $Z$ behaves like a critical Galton--Watson
process with one immigrant and ${\rm Geom}(1/2)$ offspring
distribution. Only the particles of generation $S_i-1$ for
$i\in\N$ as well as the immigrants arriving in this generation
reproduce according to ${\rm Geom}(\l_i)$ distribution. Averaging
over $\omega$ and taking into account the structure of the environment
we obtain
\begin{equation}\label{eq:connection_with_bpre}
\sum_{j=0}^{S_n}
U_j^{(S_n)}~\od~\sum_{k=1}^{S_n}Z_k\quad\text{and}\quad
S_n+\sum_{j=0}^{S_n} U_j^{(S_n)}~\od~S_n+\sum_{k=1}^{S_n}Z_k,\quad
n\in\N
\end{equation}
under the annealed probability $\P$. This leads to the most
important conclusion of the present section
\begin{equation}\label{repr}
T_{S_n}~\od~ S_n+2\sum_{k=1}^{S_n} Z_k+O_\P(1), \quad n\in\N,
\end{equation}
where $O_\P(1)$ is a term which is bounded in probability.
Distributional equality \eqref{repr} will prove useful on many
occasions.

\subsection{Notation}\label{subsec:notation}

Before we explain the strategy of our proof some more notation
have to be introduced. Denote by $Z(k,n)$ the number of progeny
residing in the $n$th generation of the $k$th immigrant. In
particular, $Z(k,k)$ is the number of offspring of this immigrant.
Then $$Z_n = \sum_{k=1}^n Z(k,n).$$ For $n\in\N$ and $1\leq i\leq
n$, let $Y(i,n)$ denote the number of progeny in the generations
$i,i+1,\ldots, n$ of the $i$th immigrant, that is,
$$Y(i,n) = \sum_{k=i}^n Z(i,k).$$ Similarly, for $i\in\N$, we denote by $Y_i$ the total progeny of the $i$th
immigrant, that is,
$$Y_i=Y(i,\infty)=\sum_{k\geq i} Z(i,k).$$ We also define $W_n$ to be the total population size in the first $n$ generations, that
is, $$W_n=\sum_{j=1}^n Z_j,\quad n\in\mathbb{N}.$$ Motivated by
the structure of the environment we shall often divide the
population into blocks which include generations $1,\ldots, S_1$;
$S_1+1,\ldots, S_2$ and so on. As a preparation, we write $$\Z_n =
Z_{S_n},\quad n\in\N$$ for the number of particles in the
generation $S_n$, $$\W_n = W_{S_n} - W_{S_{n-1}}=
\sum_{j=S_{n-1}+1}^{S_n}Z_j,\quad n\in\N$$ for the total
population in the generations $S_{n-1}+1,\ldots, S_n$ and
$$\Y_n =\sum_{j=S_{n-1}+1}^{S_n} Y_j,\quad n\in\N$$ for the total
progeny of immigrants arriving in the generations $S_{n-1},\ldots,
S_n-1$.

\subsection{Analysis of the environment}\label{sec:rde}

The asymptotic behavior of the branching process $Z$ depends
heavily upon the environment. At the end of this section we
specify qualitatively two aspects of this dependence. A random
difference equation which arises naturally in the course of our
discussion, as well as in \cite{kesten1975limit} and many other
papers on RWRE, plays an important role in the subsequent
arguments.

We proceed by recalling the definitions of random difference
equations and perpetuities. Let $(A_n, B_n)_{n\in\N}$ be a
sequence of independent copies of an $\R^2$-valued random vector
$(A, B)$. Further, let $R_0$ be a random variable which is
independent of $(A_n, B_n)_{n\in\N}$. The sequence
$(R_k)_{k\in\N_0}$, recursively defined by the random difference
equation $$R_k :=B_k+A_kR_{k-1},\quad k\in\N,$$ forms a Markov
chain which is very well known and well understood. Assuming that
$R_0=0$ and reversing the indices in an equivalent representation
$R_k=A_1\cdot\ldots\cdot A_{k-1}B_1+ A_2\cdot\ldots\cdot
A_{k-1}B_2+\ldots +B_k$ leads to the random variable
$R_k^\ast:=B_1+A_1B_2+\ldots+A_1\cdot\ldots\cdot A_{k-1}B_k$
satisfying $R_k^\ast \od R_k$ for all $k\in\N$. Whenever
\begin{equation}\label{eq:conv}
\text{the series}~\sum_{j\geq 1}B_j \prod_{l=1}^{j-1}A_l~
\text{converges a.s.}
\end{equation}
its infinite version $R_\infty^\ast:=\sum_{j\geq 1}B_j
\prod_{l=1}^{j-1}A_l$ is called perpetuity because of a possible
actuarial application. The study of the random difference
equations and perpetuities has a long history going back to
Kesten~\cite{kesten1973random} and
Grincevi\v{c}ius~\cite{Grincevicius74}. We refer the reader to the
recent monographs~\cite{buraczewski2016stochastic,iksanovrenewal}
containing a comprehensive bibliography on the subject.

It is well-known that conditions $\E \log |A| \in [-\infty, 0)$
and $\E \log^+ |B|<\infty$ are sufficient for \eqref{eq:conv} and
the distributional convergence $R_k\dod R^\ast_\infty$ as
$k\to\infty$. There are numerous results in the literature
concerning the tail behavior of $R^\ast_\infty$. The first
assertion of this flavor is the celebrated theorem by Kesten~
\cite{kesten1973random} (see also Goldie~
\cite{goldie1991implicit} and Grincevi\v{c}ius~
\cite{Grincevicius75}), to be referred to as the
Kesten-Grincevi\v{c}ius-Goldie theorem. It states that the
distribution of $R^\ast_\8$ has a heavy right tail under the
assumptions $A>0$ a.s., $\E A^s=1$ for some $s>0$ and some
additional conditions, see formula \eqref{lem:kes} below for more
details in the particular case $(A,B)=(\rho,\xi)$. The tail
behavior of $R^\ast_\infty$ is also well understood in some other
cases, in particular, when $\P\{|B|>x\}$ is regularly varying at
$\infty$ (see, for instance, \cite{Grincevicius75},
\cite{grey1994regular} and \cite{damek2017stochastic}).

Now we switch attention from the general random difference
equations to a particular one which features in the analysis of
BPRE $Z$. Using the branching property one easily obtains the
following recurrence
\begin{equation*}
\bar R_0:=\E_\omega \Z_0=0,\quad \bar R_k  := \E_\omega
\Z_k=\E_\omega Z_{S_k} = \rho_k \xi_k + \rho_k \E_{\omega}
Z_{S_{k-1}}=  \rho_k \xi_k + \rho_k \bar R_{k-1},\quad k\in\N.
\end{equation*}
This shows, among others, that the Markov chain $(\bar
R_k)_{k\in\N_0}$ is an instance of the random difference equation
which corresponds to $(A,B)=(\rho, \rho\xi)$. Asymptotic
distributional properties of a particular perpetuity which
corresponds to $(A,B)=(\rho,\xi)$ are essentially used in the
proof of Lemma \ref{lem:new}.

\section{Proof strategy}\label{sec:strategy}

A weak convergence result for $T_n$, properly normalized and
centered, will be derived from the corresponding result for
$T_{S_n}$, again properly normalized and centered. In view of
\eqref{repr}, the latter may in principle be affected by the
asymptotic behavior of $S_n$, $W_{S_n}$ or both. Fortunately, the
contribution of $S_n$ is degenerate in the limit, for it is only
regulated by the law of large numbers, fluctuations of $S_n$
around its mean do not come into play. Summarizing, analysis of
the asymptotics of $W_{S_n}$ is our dominating task.

While dealing with $W_{S_n}$ our main arguments follow the
strategy invented by Kesten et al.~\cite{kesten1975limit}. Namely,
for large $n$ we decompose $W_{S_n}$ as a sum of random variables
which are iid under the annealed probability $\P$. For this
purpose we define extinction times
\begin{equation}\label{eq:tau_def}
\tau_0:=0,\quad \tau_{k} := \min \{j>\tau_{k-1} : \Z_j = 0\},\quad
k\in\mathbb{N}.
\end{equation}
Let us emphasize that the extinctions of $Z$ are ignored in the
generations other than $S_1$, $S_2,\ldots$ Set $$\bar
\W_{\tau_n}:=W_{S_{\tau_n}} - W_{S_{\tau_{n-1}}},\quad n\in\N$$
and note that $(\bar \W_{\tau_n},\tau_n-\tau_{n-1})_{n\in\N}$ are
iid random vectors. We have
\begin{equation}\label{eq:main_two_sided_estimate}
\sum_{k=1}^{\tau^{\ast}_n}\bar \W_{\tau_k}\leq \sum_{k=1}^{S_n} Z_k\leq
\sum_{k=1}^{\tau^{\ast}_n+1}
\bar \W_{\tau_k},
\end{equation}
where $\tau_n^\ast$ is the number of extinctions of $Z$ in the
generations $S_0,\ldots, S_n$, that is,
$$\tau^{\ast}_n:=\max \{k\geq 0:\tau_k\leq n\},\quad
n\in\mathbb{N}.$$ It turns out that the extinctions occur
relatively often as the following lemma confirms.
\begin{lem}\label{lem:nu}
Assume that $\E\log \rho\in [-\infty, 0)$ and $\E\log \xi
<\infty$. Then $\E\tau_1 <\infty$. If additionally $\E
\rho^\varepsilon<\infty$ and $\E \xi^\varepsilon<\infty$ for some
$\varepsilon>0$, then $\me \exp (\gamma \tau_1)<\infty$ for some
$\gamma>0$.
\end{lem}
The proof of Lemma \ref{lem:nu} is given in the Appendix.

Under the assumptions of our main results $\mu:=\E\tau_1<\infty$
by Lemma \ref{lem:nu}. The strong law of large numbers for renewal
processes makes it plausible that, for large $n$,
\begin{equation}\label{eq:approximation}
W_{S_n}~ \approx~ \sum_{k=1}^{\lfloor \mu^{-1}n\rfloor} \bar
\W_{\tau_k}.
\end{equation}
The right-hand side, properly centered and normalized, converges
in distribution if, and only if, the distribution of $\bar
\W_{\tau_1}$ belongs to the domain of attraction of a stable law.
To check the latter, for $i\in\N$, we divide particles residing in
the generations $S_{i-1}+1,\ldots, S_i$ into groups:
\begin{itemize}
\item $\mathcal{P}_{1,i}$~--~the progeny residing in the generations $S_{i-1}+1,\ldots, S_i-1$ of the immigrants arriving
in the generations $S_{i-1},\ldots, S_i-2$, the number of these being
$$\W^0_i:=\sum_{j=S_{i-1}+1}^{S_i-1}\sum_{k=j}^{S_i-1} Z(j,k);$$
\item $\mathcal{P}_{2,i}$~--~ the progeny residing in the generations $S_{i-1}+1,\ldots, S_i-1$ of the immigrants arriving in the generations $0,1,\ldots, S_{i-1}-1$,
the number of these being $$\W^\downarrow_i:=
\sum_{j=1}^{S_{i-1}}\sum_{k=S_{i-1}+1}^{S_i-1} Z(j,k);$$
\item $\mathcal{P}_{3,i}$~--~ particles of the generation $S_i$, the number of these being $\Z_i$.
\end{itemize}
The aforementioned partition of the population which is depicted
on Figure \ref{fig:population} induces the following
decompositions
$$\W_i = \W^0_i +  \W^\downarrow_i +\Z_i,\quad i\in\N\quad \text{a.s.}$$
and $$\bar \W_{\tau_1}=\sum_{i=1}^{\tau_1} \W^0_i
+\sum_{i=1}^{\tau_1} \W^\downarrow_i +\sum_{i=1}^{\tau_1}\Z_i\quad
\text{a.s.}$$ which are of primary importance for what follows.

\begin{figure}
\centering
\begin{tikzpicture}
\draw (0,6) node[right=5pt]{$S_0$}; \draw (0,6) -- (3,0) --
(-3,0)-- (0,6); \draw[fill=gray!50] (0,6) -- (1,4) -- (-1,4) --
(0,6) node[below=26pt]{$\mathcal{P}_{1,1}$}; \draw[fill=gray!50]
(1,4) -- (2,2) -- (0,2) -- (1,4)
node[below=26pt]{$\mathcal{P}_{1,2}$}; \draw[fill=gray!50] (2,2)
-- (3,0) -- (1,0) -- (2,2) node[below=26pt]{$\mathcal{P}_{1,3}$};
\draw[line width=1mm] (-1,4) -- (1,4) node[right=5pt]{$S_1$};
\draw[line width=1mm] (-2,2) -- (2,2) node[right=5pt]{$S_2$};
\draw[line width=1mm] (-3,0) -- (3,0) node[right=5pt]{$S_3$};
\draw (-0.5,1) node{$\mathcal{P}_{2,3}$}; \draw (-0.5,3)
node{$\mathcal{P}_{2,2}$};
\end{tikzpicture}
\caption{The generations $0$ through $S_3$ of the BPRE $Z$ and the
partition of the corresponding population into parts
$\mathcal{P}_{i,j}$, $i,j=1,2,3$. The bold horizontal lines
represent particles in the generations $S_1$, $S_2$ and $S_3$,
that is, those comprising the groups $\mathcal{P}_{3,i}$,
$i=1,2,3$. By definition, $\mathcal{P}_{2,1}=\oslash$.}
\label{fig:population}
\end{figure}
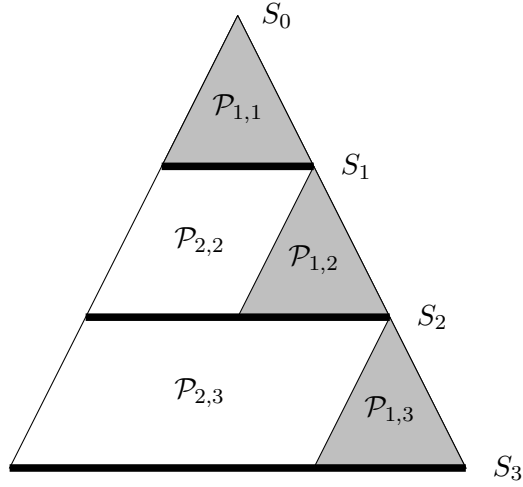
Depending on the assumptions ($\rho 1$), ($\rho 2$), ($\xi 1$) or
($\xi 2$) the random variables $\sum_{i=1}^{\tau_1}\W^0_i$,
$\sum_{i=1}^{\tau_1}\W^\downarrow_i$ and $\sum_{i=1}^{\tau_1}\Z_i$
may exhibit different tail behaviors. Often, one of the random
variables dominates the others thereby determining the tail
behavior of the whole sum $\bar \W_{\tau_1}$.

\section{Tail behavior of \texorpdfstring{$\bar{\W}_{\tau_1}$}{Wtau}}\label{sec:tails}

In this section we do not assume that $\E\xi<\infty$.

We first analyze the tail behavior of $\sum_{i=1}^{\tau_1}\W^0_i$.
Note that by construction $(\W^0_i)_{i\in\N}$ are iid and the
random variable $\tau_1$ {\it does not depend on the future of the
sequence } $(\W^0_i)_{i\in\N}$ in the sense of the definition
given by Denisov, Foss, Korshunov on p.~987 in
\cite{denisov:foss:korshunov}. The latter means that, for each
$n\in\N$, the collections of random variables $((\W^0_k)_{k\le n},
{\bf 1}_{\{\tau_1 \le n\}})$ and $(\W^0_k)_{k
> n}$ are independent. This observation in combination with Corollary 3 in \cite{denisov:foss:korshunov} suggests that it is enough
to analyze the tail behavior of just one summand, $\W^0_1$ say,
provided that the right tail of $\W^0_1$ is regularly varying and
heavier than the right tail of $\tau_1$.
\begin{lem}
\label{lem:lu1}
Assume that \eqref{eq:reg_var_xi} holds with some $\beta>0$. Then $$\P\{\W^0_1 >x\}~\sim~
\E\vartheta^{\beta/2} x^{-\beta/2}\ell(x^{1/2}),\quad x\to\infty,$$ where $\vartheta$ is a random variable with Laplace transform
\begin{equation}\label{eq:lt_limit}
\E e^{-s\vartheta}=1/\cosh(s^{1/2}),\quad s\geq 0.
\end{equation}
\end{lem}
The proof of Lemma \ref{lem:lu1} is given in Section \ref{sec:GW}.
In the next two lemmas we provide moment estimates for the two
other summands $\sum_{i=1}^{\tau_1}\W^\downarrow_i$ and
$\sum_{i=1}^{\tau_1}\Z_i$.
\begin{lem}\label{lem:zmom}
Assume that $\E\log\rho\in [-\infty, 0)$ and that, for some
$\kappa \le 2$, $\E (\rho \xi)^{\kappa}$ and $\E \xi^\kappa$ are
finite. Then $\E \Z_1^\kappa<\infty$ and there exists a positive
constant $C$ such that, for all $n\in\N$,
\begin{equation}\label{eq:5}
\E \Z_n^\kappa \leq \left\{  \begin{array}{cc}
C& \mbox{ if } \gamma <1,\\
Cn& \mbox{ if } \gamma =1,\\
C\gamma^n& \mbox{ if } \gamma >1,
\end{array}\right.
\end{equation}
where $\gamma:=\E\rho^\kappa$. If additionally $\E
\xi^{2\kappa}<\8$, then
\begin{equation} \label{eq:lu1}
\E \W_1^\kappa <\8.
\end{equation}
\end{lem}
\begin{rem}
Since $\xi\geq 1$ a.s., the assumption $\E (\rho
\xi)^\kappa<\infty$ entails $\E \rho^\kappa<\infty$. This explains
the absence of the latter condition in Lemma \ref{lem:zmom}.
\end{rem}
\begin{lem}\label{lem:Wnu}
Assume that, for some $\kappa \le 2$, $\E \rho^\kappa<1$, $\E
(\rho \xi)^\kappa$ and $\E \xi^\kappa$ are finite. Then, for all
$\kappa_0 \in (0,\kappa)$,
\begin{equation}\label{5??}
\E  \left(\sum_{i=1}^{\tau_1}\Z_i  \right)^{\kappa_0} <
\infty.
\end{equation}
If additionally $\E\xi^{3\kappa/2}<\infty$, then
\begin{equation}\label{eq:5?}
\E\left(\sum_{i=1}^{\tau_1}\W_i^\downarrow\right)^{\kappa_0}<\infty.
\end{equation}
\end{lem}

Lemma \ref{prop:main2} states that under the assumption $(\rho 1)$
the distribution of $\sum_{k=1}^{\tau_1} \big(
\Z_k+\W_k^\downarrow \big)$ has a power tail.
\begin{lem} \label{prop:main2}
Assume that $(\rho 1)$ holds for some $\alpha\in (0,2]$,
$\E\xi^{3\alpha/2}<\infty$ and $\E(\rho\xi)^\alpha<\infty$. Then
$$\P\bigg\{ \sum_{k=1}^{\tau_1} \big( \Z_k + \W_k^\downarrow  \big)
>x \bigg\} \sim C_2(\alpha) x^{-\a},\quad x\to\infty$$
for a positive constant $C_2(\alpha)$.
\end{lem}

Lemma \ref{lem:1} points out the tail behavior of $\bar
\W_{\tau_1}$ in the situation where the slowly varying factor in
$(\xi 2)$ is a constant.
\begin{lem}\label{lem:1}
Assume that $(\rho 1)$ holds for some $\alpha\in (0,2)$, $(\xi 2)$
holds with $\beta=2\alpha$ and $\ell$ such that
$\lim_{t\to\infty}\ell(t)=C_{\ell}>0$,
$\E\rho^{\alpha+\varepsilon}<\infty$ and $\E\rho^\alpha
\xi^{\alpha+\varepsilon}<\infty$ for some $\varepsilon>0$. Then
$$\P\{\bar \W_{\tau_1} >x \} \sim \left((\E\tau_1)(\E \vartheta^{\alpha})C_{\ell}+C_2(\alpha)\right) x^{-\alpha},\quad x\to\infty,$$
where $C_2(\alpha)$ is the same constant as in Lemma
\ref{prop:main2}.
\end{lem}

The proofs of Lemmas \ref{lem:zmom} through \ref{lem:1} are
postponed until Section \ref{sec:proofs_lemmas}.

For the ease of reference the tail behavior of $\bar \W_{\tau_1}$
is summarized in the following proposition.
\begin{prop}\label{prop:tail_main}
The following asymptotic relations hold.
\begin{itemize}
\item[{\rm (C1)}] If $(\rho 1)$ holds for some $\alpha\in(0,2]$, either
$\E\xi^{2\alpha}<\infty$ or $(\xi 2)$ holds with $\beta=2\alpha$,
$\lim_{t\to\infty}\ell(t)=0$, and $\E(\rho\xi)^\alpha <\infty$,
then $$\P\{\bar \W_{\tau_1}>x\}\sim C_2(\alpha)x^{-\alpha},\quad
x\to\infty,$$ where $C_2(\alpha)$ is the same constant as in Lemma
\ref{prop:main2}.
\item[{\rm (C2)}] If $(\rho 1)$ holds for some $\alpha\in(0,2)$, $(\xi 2)$ holds with $\beta=2\alpha$ and
$\lim_{t\to\infty}\ell(t)=C_{\ell}\in(0,\infty)$,
$\E\rho^{\alpha+\varepsilon}<\infty$ and $\E\rho^\alpha
\xi^{\alpha+\varepsilon}<\infty$ for some $\varepsilon>0$, then
$$\P\{\bar \W_{\tau_1}>x\}\sim \left((\E\tau_1)(\E \vartheta^{\alpha})C_{\ell}+C_2(\alpha)\right) x^{-\alpha},\quad
x\to\infty.$$
\item[{\rm (C3)}] If $(\rho 1)$ holds for some $\alpha\in(0,2]$, $(\xi 2)$ holds with $\beta=2\alpha$ and
$\lim_{t\to\infty}\ell(t)=\infty$, and
$\E(\rho\xi)^\alpha<\infty$, then
$$\P\{\bar \W_{\tau_1}>x\}\sim (\E\tau_1)(\E \vartheta^{\alpha})x^{-\alpha}\ell(x^{1/2}),\quad x\to\infty.$$
\item[{\rm (C4)}] If $(\rho 2)$ holds, $(\xi 2)$ holds for some $\beta\in(0,4)$ such that $\beta/2\in\mathcal{I}$ and
$\E (\rho\xi)^{\beta/2+\varepsilon}<\infty$ for some
$\varepsilon>0$, then $$ \P\{\bar \W_{\tau_1}>x\}\sim
(\E\tau_1)(\E \vartheta^{\beta/2})x^{-\beta/2}\ell(x^{1/2}),\quad
x\to\infty.$$
\end{itemize}
\end{prop}
\begin{proof}
Under the assumptions ${\rm (Ci)}$, $i=1,2,3,4$ $\tau_1$ has some
finite exponential moment by Lemma \ref{lem:nu}. This fact will be
used when applying Corollary 3 of \cite{denisov:foss:korshunov}
below.

\noindent {\sc Proof of (C1).} Each of $\E\xi^{2\alpha}<\infty$
and $(\xi 2)$ with $\beta=2\alpha$ implies $\E
\xi^{3\alpha/2}<\infty$. Therefore, in view of Lemma
\ref{prop:main2} it is enough to show that
\begin{equation}\label{eq:cor_main1}
\P\bigg\{\sum_{i=1}^{\tau_1}\W_i^0>x\bigg\}=o(x^{-\alpha}),\quad x\to\infty.
\end{equation}
Corollary 3 in \cite{denisov:foss:korshunov} ensures that
\begin{equation}\label{asymp}
\P\bigg\{\sum_{i=1}^{\tau_1}\W_i^0>x\bigg\}\sim (\E
\tau_1)\P\{\W_1^0>x\},\quad x\to\infty
\end{equation}
whenever the right tail of $\W_1^0$ is regularly varying.

If $(\xi 2)$ holds with $\beta=2\alpha$, then according to Lemma
\ref{lem:lu1}
$$\P\{\W_1^0>x\} ~\sim~  \E\vartheta^\alpha
x^{-\alpha}\ell(x^{1/2}),\quad x\to\infty.$$ This in combination
with $\lim_{t\to\infty}\ell(t)=0$ which holds by assumption and
\eqref{asymp} proves \eqref{eq:cor_main1}.

Assuming that $\E\xi^{2\alpha}<\infty$ we intend to show that
\begin{equation}\label{mom1}
\E \Big[\sum_{i=1}^{\tau_1}\W_i^0\Big]^\alpha<\infty
\end{equation}
which, of course, entails \eqref{eq:cor_main1}. By Lemma
\ref{lem:future}, \eqref{mom1} holds provided that $\E
[\W_1^0]^\alpha <\infty$. The latter is secured by
$\E\xi^{2\alpha}<\infty$ and Lemma \ref{lem:mom}.

\noindent {\sc Proof of (C2).} This is just Lemma \ref{lem:1}.

\noindent {\sc Proof of (C3).} This follows from Lemma
\ref{lem:lu1} in conjunction with Corollary 3 in
\cite{denisov:foss:korshunov} and Lemma \ref{prop:main2} because
$(\xi 2)$ with $\beta=2\alpha$ entails $\E\xi^{3\alpha/2}<\infty$.

\noindent {\sc Proof of (C4)}. Since the interval $\mathcal{I}$ is
open, there exists $\varepsilon_1>0$ such that $\beta/2+\varepsilon_1\in (0,2]$,
$\E\rho^{\beta/2+\varepsilon_1}<1$,
$\E\xi^{3\beta/4+3\varepsilon_1/2}<\infty$ and
$\E(\rho\xi)^{\beta/2+\varepsilon_1}<\infty$. In view of this
Lemma \ref{lem:Wnu} applies with $\kappa=\beta/2+\varepsilon_1$
and $\kappa_0=\beta/2+\varepsilon_1/2$ which gives $\E  \big(\sum_{i=1}^{\tau_1}\Z_i  \big)^{\beta/2+\varepsilon_1/2} <
\infty$ and $\E\big(\sum_{i=1}^{\tau_1}\W_i^\downarrow\big)^{\beta/2+\varepsilon_1/2}<\infty$. An appeal to Lemma \ref{lem:lu1} in
combination with Corollary 3 in \cite{denisov:foss:korshunov} does the rest.
\end{proof}

\section{Critical Galton--Watson process with immigration}\label{sec:GW}

As has already been mentioned in Section~\ref{sec:bpre},
$(Z_n)_{0\leq n\leq \xi_1-1} \od (\Zc_n)_{0\leq n\leq \xi_1-1}$,
where $\xi_1$ is assumed independent of $(\Zc_n)_{n\in\N_0}$ a
critical Galton--Watson process with unit immigration and ${\rm
Geom}(1/2)$ offspring  distribution. In this section we collect
some known properties of $(\Zc_n)_{n\in\N_0}$ and prove several
auxiliary results which to our knowledge are not available in the
literature. The evolution of $(\Zc_n)_{n\in\N_0}$ is the same as
that of the BRPE $Z$ with $\omega_n\equiv 1/2$ for all $n\in\N$,
see Section \ref{sec:bpre1}.

For $n\in\N$, let $\Wc_n:=\sum_{k=1}^n \Zc_k$ denote the total
progeny in the first $n$ generations. Further, for $n\in\N$ and
$1\leq k\leq n$, write $\Zc(k,n)$ for the number of the $n$th
generation progeny of the $k$th immigrant and $\Yc(k,n)$ for the
number of progeny of the $k$th immigrant which reside in
generations $k$ through $n$, that is,
$$\Yc(k,n) =\sum_{j=k}^n \Zc(k,j).$$

Here is the main result of this section of which Lemma
\ref{lem:lu1} is an immediate consequence because $\W_1^0\od
\Wc_{\xi_1-1}$, where $\xi_1$ is assumed independent of
$(\Wc_k)_{k\in\N}$.
\begin{prop}\label{prop:GW}
Let $\varsigma$ be an integer-valued random variable independent
of $(\Wc_n)_{n\in\N_0}$ and such that $$ \P\{\varsigma>x\}\sim
x^{-2\alpha}\ell(x),\quad x\to\infty$$ for some $\alpha>0$ and
some $\ell$ slowly varying at $\infty$. Then $$
\P\{\Wc_\varsigma>x\}~\sim~
\E\vartheta^{\alpha}\P\{\varsigma>x^{1/2}\} ~\sim~
\E\vartheta^{\alpha} x^{-\alpha}\ell(x^{1/2}),\quad x\to\infty,$$
where $\vartheta$ is a random variable with Laplace transform
\eqref{eq:lt_limit}.
\end{prop}
\begin{rem}
For fixed $n\in\N$, $\E{\Wc_n} = \frac{n(n+1)}{2}$ and the
distribution of $\Wc_n$ inherits an exponential tail from ${\rm
Geom}(1/2)$ offspring distribution. Thus, for $\varsigma$ which
has distribution with a heavy tail and is independent of
$(\Wc_n)_{n\in\N}$ it is natural to expect that
$$\Wc_\varsigma \approx {\rm const}\cdot \varsigma^2.$$ Proposition
\ref{prop:GW} makes this intuition precise.
\end{rem}
Lemma \ref{lem:mom} given next is used in the proof of Proposition
\ref{prop:tail_main}, part ${\rm (C1)}$.
\begin{lem}\label{lem:mom}
Let $\varsigma$ be an integer-valued random variable independent
of $(\Wc_n)_{n\in\N_0}$ and such that $\E
\varsigma^{2\alpha}<\infty$ for some $\alpha>0$. Then $\E
[\Wc_\varsigma]^\alpha<\infty$.
\end{lem}
To prove Proposition \ref{prop:GW} and Lemma \ref{lem:mom} we need
some auxiliary lemmas. The first one is due to Pakes~\cite[Theorem
5]{pakes1972critical}.
\begin{lem}
We have
\begin{equation}\label{distr_conv}
n^{-2} \Wc_n~\dod~ \vartheta,\quad n\to\infty,
\end{equation}
where $\vartheta$ is a random variable with Laplace transform \eqref{eq:lt_limit}.
\end{lem}

In the cited article Pakes investigates Galton--Watson processes
with general, not necessarily unit, immigration. One of the
standing assumptions of that paper is that the probability of
having no immigrants is positive. However, a perusal of the proof
of Theorem 5 in \cite{pakes1972critical} reveals that the result
still holds without this assumption.

With some additional effort one can prove the convergence of all
moments in \eqref{distr_conv}.
\begin{lem}\label{lem:exp_moments}
For each $s>0$,
\begin{equation}\label{conv_mom}
\lim_{n\to\infty}\E(n^{-2}\Wc_n)^s = \E \vartheta^s.
\end{equation}
\end{lem}
\begin{proof}
Suppose for the moment that we have verified that
\begin{equation}\label{bound}
\sup_{n\geq n_0}\E\exp(\beta n^{-2}\Wc_n)<\infty
\end{equation}
for some $\beta>0$ and some $n_0\in\NN$. Then in view of
$$\sup_{n\geq n_0} \E(n^{-2}\Wc_n)^s\leq  C(s) \sup_{n\geq n_0}\E\exp(\beta n^{-2}\Wc_n)<\infty$$
for all $s>0$ and some constant $C(s)$, the Vall\'{e}e--Poussin
criterion for uniform integrability in combination with
\eqref{distr_conv} ensures \eqref{conv_mom}.

Left with the proof of \eqref{bound} observe that, for fixed
$k\in\N$, the process initiated by the $k$th immigrant
$(\Zc(k,n))_{n\ge k}$ is a Galton--Watson process with ${\rm
Geom}(1/2)$ offspring distribution. Moreover, the processes
started by different immigrants are iid. Therefore, writing
\begin{align*}
\Wc_n & =\sum_{k=1}^{n} \Zc_k = \sum_{k=1}^n \sum_{j=1}^k \Zc(j,k)\\
&= \sum_{j=1}^n \bigg(  \sum_{k=j}^n \Zc(j,k)  \bigg)=
\sum_{j=1}^n \Yc(j,n)\quad \text{a.s.}
\end{align*}
we obtain a representation of $\Wc_n$ as the sum of independent random variables. This formula entails
\begin{equation}\label{eq:z_n_repp}
\E\exp\big(x \Wc_n\big)=\prod_{j=1}^n a_j(x),\qquad  x\ge 0
\end{equation}
(the case that both sides of \eqref{eq:z_n_repp} are infinite for some $x>0$ is not excluded), where
$$a_j(x):=\E \exp \big( x \Yc(n-j+1,n)\big) = \E \exp \big( x \Yc(1,j)\big),\quad 1\leq j\leq n, x\geq 0.$$
We have $a_0(x) =1$ for all $x\geq 0$ and
\begin{equation*}
a_1(x) = \E\exp\big( x\Zc(1,1)  \big) = \sum_{k\geq
0}e^{kx}2^{-k-1}=(2-e^x)^{-1}
\end{equation*}
for $x\in [0,\log 2)$. Using a decomposition
\begin{equation}\label{repr2}
\Yc(1,j) = \sum_{m=1}^{\Zc(1,1)}\Yc_m(1,j-1)  + \Zc(1,1),\quad
j\geq 2\quad \text{a.s.},
\end{equation}
where $(\Yc_m(1,j-1) )_{m\in\N}$ are independent copies of
$\Yc(1,j-1)$ which are also independent of $\Zc(1,1)$ we infer
$$a_j(x) = \frac {1}{2-e^xa_{j-1}(x)},\quad j\in\N.$$ In
particular, for every fixed $j\in\N_0$, $a_j(x)<\infty$ for all
$x$ from some right vicinity of the origin.

Set $b_j(x) = e^x a_j(x)$ for $j\in \N_0$ and $x\geq 0$, so that
$$b_j(x) = \frac {e^x}{2-b_{j-1}(x)}.$$ By technical reasons, it is more convenient to work with $b_j$
rather than $a_j$. We intend to show that, for every
$\gamma\in(0,1/4)$, there exists $K=K(\gamma)>1$ and
$x_0(\gamma)>0$ such that
\begin{equation}\label{eq:a_j_up_bound}
b_j(x)\leq 1+Kx(j+1).
\end{equation}
for $j\in \N_0$ and $x>0$ satisfying $j(1+j)x\leq \gamma$ and
$x<x_0(\gamma)$.

Given $\gamma\in(0,1/4)$ pick $K>1$ such that $K-K^2\gamma>1$.
This is possible because the largest root of the quadratic
equation $\gamma x^2-x+1=0$ is larger than one. There exists
$x_0(\gamma)>0$ such that
$$e^x \leq 1+(K-K^2\gamma)x,\quad x\in(0,x_0(\gamma)).$$ Moreover, since we assume $j(1+j)x\leq \gamma$ we have
$$e^x \leq 1+Kx - K^2x^2 j(j+1)=(1-Kxj)(1+Kx(j+1)).$$ Now \eqref{eq:a_j_up_bound} follows by the mathematical induction.
While for $j=0$ we obtain $$b_0(x)=e^x\leq 1+(K-K^2\gamma)x\leq
1+Kx,\quad x\in(0,x_0(\gamma)),$$ an induction step works as
follows $$b_j(x)=\frac{e^x}{2-b_{j-1}(x)}\leq
\frac{e^x}{1-Kjx}\leq 1+Kx(j+1)$$ for $x\in(0,x_0(\gamma))$ and
$j(j+1)x\leq\gamma$. The proof of~\eqref{eq:a_j_up_bound} is
complete.

Armed with~\eqref{eq:a_j_up_bound} we can deduce \eqref{bound}.
Given $\beta \in (0,1/4)$ take $\gamma\in(\beta,1/4)$ and pick
$n_0\in\NN$ such that $\beta/n^2<x_0(\gamma)$ and $(n+1)\beta\leq
n\gamma$ for $n \geq n_0$. Such a choice ensures that $j(j+1)\beta
n^{-2}\leq \gamma$ for integer $0\leq j\leq n$ whenever $n\geq
n_0$. Using \eqref{eq:z_n_repp} and then \eqref{eq:a_j_up_bound}
we arrive at $$\E \exp(\beta n^{-2} \Wc_n)=\prod_{j=0}^n
a_j\left(\beta n^{-2}\right)\leq \prod_{j=0}^n b_j\left(\beta
n^{-2}\right)\leq \prod_{j=0}^n(1+K\beta n^{-2}(j+1)),\quad n\geq
n_0$$ for $\beta\in(0,1/4)$. It remains to note that $$\sup_{n\geq
n_0} \prod_{j=0}^n(1+K\beta n^{-2}(j+1))\leq
\exp(3K\beta)<\infty,$$ thereby finishing the proof of
\eqref{bound}.
\end{proof}

We are now ready to prove Proposition \ref{prop:GW} and Lemma
\ref{lem:mom}.
\begin{proof}[Proof of Proposition \ref{prop:GW}]
By virtue of \eqref{distr_conv} we infer $\Wc_n\to\infty$ in probability and then $\Wc_n\to\infty$ a.s.\ by monotonicity. Therefore,
$$\upsilon_x:=\inf\{k\in\NN:\Wc_k >x\}\in [1,\infty)\quad\text{a.s}\quad \text{for } x>1.$$
For $x>1$ we have
$$\P\{\Wc_\varsigma>x\}=\P\{\varsigma\geq \upsilon_x\}=\E
h(\upsilon_x),$$ where $h(y):=\P\{\varsigma\geq y\}$. Under the
introduced notation, we have to prove that
\begin{equation}\label{conv_mom2}
\lim_{x\to\infty}\frac{\E h(\upsilon_x)}{h(x^{1/2})}
=\E\vartheta^\alpha.
\end{equation}
By a standard inversion technique \`{a} la Feller (see Theorem 7 in \cite{feller1949}) \eqref{distr_conv} entails
\begin{equation}\label{distr2}
\frac{\upsilon_x}{x^{1/2}}~\dod~ \vartheta^{-1/2},\quad
x\to\infty.
\end{equation}
We claim that the latter implies further that
\begin{equation}\label{distr_conv2}
\frac{h(\upsilon_x)}{h(x^{1/2})}~\dod~ \vartheta^{\alpha},\quad
x\to\infty.
\end{equation}
The simplest way to see it is to pass in \eqref{distr2} to versions which converge a.s., that is,
$$\lim_{x\to\infty}x^{-1/2}\upsilon_x^\ast=(\vartheta^\ast)^{-1/2}\quad\text{a.s.}$$
and then exploit the fact that
$$\lim_{x\to\infty}\frac{h(y(x)x^{1/2})}{h(x^{1/2})}=y^{-2\alpha}\quad\text{whenever}\quad
\lim_{x\to\infty}y(x)=y\in (0,\infty)$$ (see Theorem 1.5.2 in
\cite{bingham1989regular}). This gives
$$\lim_{x\to\infty}\frac{ h((x^{-1/2}\upsilon^\ast_x)x^{1/2})}{h(x^{1/2})}=(\vartheta^\ast)^{\alpha}\quad\text{a.s.}$$
because $\vartheta^\ast>0$ a.s.

With \eqref{distr_conv2} at hand, relation \eqref{conv_mom2}
follows if we can show that the family
$\left(h(\upsilon_x)/h(x^{1/2})\right)_{x\geq x_0}$ is uniformly
integrable for some $x_0>0$. By Potter's bound for regularly
varying functions (Theorem 1.5.6 (iii) in
\cite{bingham1989regular}), given $A>1$ and $\delta>0$ there
exists $n_1\in\NN$ such that
$$\frac{h(\upsilon_x){\bf 1}_{\{\upsilon_x>n_1\}}}{h(x^{1/2})}
\leq A \max ((x^{-1/2}\upsilon_x)^{-2\alpha-\delta},
(x^{-1/2}\upsilon_x)^{-2\alpha+\delta})\quad\text{a.s.}$$ whenever
$x\geq n_1^2$. Further, by monotonicity of $h$,
$$\frac{h(\upsilon_x){\bf 1}_{\{\upsilon_x\leq n_1\}}}{h(x^{1/2})} \leq \frac{h(1)}{h(x^{1/2})}{\bf 1}_{\{\upsilon_x\leq n_1\}}\quad\text{a.s.}$$
Thus, for uniform integrability of $\left(h(\upsilon_x)/h(x^{1/2})\right)_{x\geq x_0}$ it suffices to check two things: first,
\begin{equation}\label{eq:t_x_unif_integrability}
\sup_{x\geq 4} x^{\beta/2}\E \upsilon_x^{-\beta}<\infty
\end{equation}
for some $\beta>2\alpha$ and second
\begin{equation}\label{ineq2}
\sup_{x\geq x_0}\left( \frac{h(1)}{h(x^{1/2})}
\right)^\gamma\P{\{\upsilon_x\leq n_1\}}<\infty
\end{equation}
for some $\gamma>1$.

From the proof of Lemma \ref{lem:exp_moments} we know that $\E
\exp(s\Wc_{n_1})<\infty$ for some $s>0$, whence
\begin{equation*}
\P\{\upsilon_x\leq n_1\}=\P\{\Wc_{n_1}>x\}=O(e^{-sx}),\quad
x\to\infty
\end{equation*}
which proves \eqref{ineq2}.

Now we intend to show that \eqref{eq:t_x_unif_integrability} holds for all $\beta> 0$. We have for $x\geq 4$
\begin{align*}
\E \upsilon_x^{-\beta}&=\int_0^{1}\P\{\upsilon_x^{-\beta}>y\}{\rm d}y=\beta\int_1^\infty\P\{\upsilon_x\leq z\}z^{-\beta-1}{\rm d}z
\leq  \beta \sum_{k\geq 2}\P\{\upsilon_x\leq k\}(k-1)^{-\beta-1}\\
&=\beta \sum_{k=2}^{[x^{1/2}]}\P\{\Wc_k>x\}(k-1)^{-\beta-1}+\beta \sum_{k\geq [x^{1/2}]+1}\P\{\Wc_k>x\}(k-1)^{-\beta-1}\\
& \leq \beta \sum_{k=2}^{[x^{1/2}]}\frac{\E
(\Wc_k)^{\beta}}{x^{\beta}(k-1)^{\beta+1}}+\beta\sum_{k\geq
[x^{1/2}]+1}\frac{1}{(k-1)^{\beta+1}}\leq \frac{{\rm
const}}{x^{\beta}}\sum_{k=1}^{[x^{1/2}]}k^{\beta-1}+O(x^{-\beta/2})=O(x^{-\beta/2}),
\end{align*}
where the last and penultimate inequalities follow from
Lemma~\ref{lem:exp_moments} and Markov's inequality, respectively.
The proof of Proposition \ref{prop:GW} is complete.
\end{proof}
\begin{proof}[Proof of Lemma \ref{lem:mom}]
By Lemma \ref{lem:exp_moments}, $\E [n^{-2}\Wc_n]^\alpha \leq C$
for all $n\in\N$ and some $C>0$. This entails
$$\E [\Wc_\varsigma]^\alpha=\sum_{n\geq 1}\E [n^{-2}\Wc_n]^\alpha n^{2\alpha} \P\{\varsigma=n\}\leq C
\E\varsigma^{2\alpha}<\infty.$$ The proof of Lemma \ref{lem:mom}
is complete.
\end{proof}

For later use, we note that, for $n\in\N$,
\begin{equation}\label{eq:harris}
\begin{split}
\E \Zc(1,n)  = 1, &\quad {\rm Var}\, \Zc(1,n) = 2n, \\
\E \Yc(1,n)  = n, &\quad {\rm Var}\, \Yc(1,n) =
\frac{n(n+1)(2n+1)}{3}.
\end{split}
\end{equation}
The first three of these equalities follow by an elementary
calculation. The fourth one can be derived with the help of
\eqref{repr2} and the mathematical induction.

\section{Proofs}\label{sec:proofs_all}

\subsection{Proof of Proposition \ref{thm:lln}}\label{sec:lln}

Recalling that $v=\me\xi /\me T_{S_1}$ it suffices to show that
$$\me T_{S_1} = \left\{  \begin{array}{cc} \E \xi^2+ \frac{2\me\xi\me\rho\xi}{1-\E \rho},
& \text{if }\E \rho<1, \: \E \rho\xi<\infty, \E \xi^2<\infty; \\
\infty, & \text{otherwise.} \end{array}\right.$$ Using
\eqref{repr} yields
$$\frac{T_{S_n}}{n}\topr \me T_{S_1},\quad n\to\infty\quad \Longleftrightarrow\quad
\frac{\sum_{j=1}^{S_n} Z_j}{n}=\frac{W_{S_n}}{n}\topr
\frac{1}{2}\left(\me T_{S_1} -\E\xi\right),\quad n\to\infty.$$ Let
us prove the latter convergence in probability. According to Lemma
\ref{lem:nu}, we have $\me \tau_1<\infty$ whenever $\me
\log\rho\in [-\infty, 0)$ and $\me\log^+ \xi<\infty$. Recalling
from \eqref{eq:main_two_sided_estimate} that
$$\frac{1}{n}\sum_{k=1}^{\tau_n^{\ast}}\bar \W_{\tau_k}\leq
\frac{W_{S_n}}{n}\leq\frac{1}{n}\sum_{k=1}^{\tau_n^{\ast}+1}\bar
\W_{\tau_k}$$ we conclude by the strong law of large numbers that
$$\lim_{n\to\infty} \frac{W_{S_n}}{n}=\frac{1}{\E \tau_1}\E\bar \W_{\tau_1}\quad \P-\text{a.s.}$$ Hence, $$\me T_{S_1}=\me \xi+ \frac{2}{\E
\tau_1}\E\bar \W_{\tau_1}.$$

Left with identifying $\E\bar \W_{\tau_1}$ we recall that, for
$k\in\N$, $\Y_k$ denotes the total progeny of immigrants arriving
in the generations $S_{k-1},\ldots, S_k-1$, that is,
$$\Y_k = \sum_{j=S_{k-1}+1}^{S_k} Y(j,\infty).$$
Since $\Y_1$, $\Y_2,\ldots$ are identically distributed and, for
$k\in\N$, $\Y_k$ is independent of $\{\tau_1\geq
k\}=\{Z_{S_1}>0,\ldots, Z_{S_{k-1}}>0\}$ we infer
\begin{align*}
\E \bar \W_{\tau_1}& = \E \sum_{k=1}^{\tau_1}\Y_k = \sum_{k\geq 1}
\E \Y_k {\bf 1}_{\{\tau_1 \geq k \}} =\sum_{k\geq 1} \E \Y_k \P\{
\tau_1 \geq k \} = \E \Y_1 \E \tau_1
\end{align*}
(if $\E \Y_1=\infty$, the formula just says that $\E \bar
\W_{\tau_1}=\infty$). To calculate $\E\Y_1$ we note that
$$\E_\omega Y(j,\infty){\bf 1}_{\{j\leq \xi_1\}} = \Big(\xi_1 -j +
\sum_{k\geq 2}\xi_k \prod_{i=1}^{k-1}\rho_i\Big){\bf 1}_{\{j\leq
\xi_1\}}\quad \text{a.s.},$$ whence
$$\E_\omega\Y_1 = \frac{\xi_1(\xi_1-1)}{2} + \xi_1\rho_1 \sum_{k\geq
2}\xi_k \prod_{i=2}^{k-1}\rho_i\quad \text{a.s.},$$ where the
a.s.\ convergence of the last series is secured by our assumptions
$\me \log\rho\in [-\infty, 0)$ and $\me\xi<\infty$. Taking the
expectation with respect to $\P$ yields
$$\E \Y_1 = \left\{  \begin{array}{cc} \frac 12 \E \xi(\xi-1) + \frac{\E\xi \E \rho\xi}{1-\E \rho},  & \text{if }\E \rho<1,
\E \rho\xi<\infty, \E \xi^2<\infty; \\
\infty, & \text{otherwise. } \end{array}\right.$$ The proof of
Proposition \ref{thm:lln} is complete.

\subsection{Proof of Theorem \ref{thm:main1T} and Corollary
\ref{cor:main1X}}\label{7.2}

The assumptions of Theorem \ref{thm:main1T} ensure that
$\E\xi<\infty$ and that $\mu:=\E \tau_1$ and $s^2:={\rm
Var}\,\tau_1$ are finite (for the latter use Lemma \ref{lem:nu}).
It is also clear that the distribution of $\tau_1$ is
nondegenerate, whence $s^2>0$.

From Proposition \ref{prop:tail_main} (parts {\rm (C1)} and {\rm
(C2)}) we know that
$$\P\{\bar \W_{\tau_1}>x\}\sim C x^{-\alpha},\quad x\to\infty,$$
where $C=C_2(\alpha)$ in the cases ${\rm (A1)}$ and ${\rm (A2)}$
and $C=(\E\tau_1)(\E \vartheta^\alpha)C_{\ell}+C_2(\alpha)$ in the
case ${\rm (A3)}$. Therefore, the distribution of $\bar
\W_{\tau_1}$ belongs to the domain of attraction of an
$\alpha$-stable distribution. This means that
\begin{equation}\label{eq:weak_limit_thm1_proof1}
\frac{\sum_{k=1}^n \bar \W_{\tau_k}-a(n)}{b(n)}~\dod~
\mathcal{S}_{\alpha},\quad n\to\infty
\end{equation}
for some $a(t)$ and $b(t)$, where $\mathcal{S}_2\od
\mathcal{N}(0,1)$. To find $a(t)$ and $b(t)$ explicitly we use
Theorem 3 on p.~580 and formula (8.15) on p.~315 in
\cite{feller:1971}:

\noindent $b(t)=(Ct)^{1/\alpha}$ and $a(t)=0$ if $\alpha\in (0,1)$;

\noindent $b(t)=Ct$ and $a(t)=t\int_0^{Ct}\P\{\bar \W_{\tau_1}>x\}{\rm
d}x$ if $\alpha=1$;

\noindent $b(t)=(Ct)^{1/\alpha}$ and $a(t)=(\E \bar \W_{\tau_1})t$ if
$\alpha\in (1,2)$;

\noindent $b(t)=(Ct\log t)^{1/2}$ and $a(t)=(\E \bar \W_{\tau_1})t$ if
$\alpha=2$.

Our subsequent proof will be based on representation \eqref{repr}.
In view of this we first analyze the asymptotics of $W_{S_n}$.

\noindent {\sc Step 1. Limit theorems for $W_{S_n}$}. We claim
that
\begin{equation}\label{eq:weak_limit_thm1_proof21}
\frac{W_{S_n}-a(\mu^{-1}n)}{b(\mu^{-1} n)}\dod
\mathcal{S}_{\alpha},\quad n\to\infty.
\end{equation}
In view of \eqref{eq:main_two_sided_estimate} relation
\eqref{eq:weak_limit_thm1_proof21} follows once we have checked
that \eqref{eq:weak_limit_thm1_proof1} entails
\begin{equation}\label{eq:weak_limit_thm1_proof2}
\frac{\sum_{k=1}^{\tau_n^{\ast}}\bar
\W_{\tau_k}-a(\mu^{-1}n)}{b(\mu^{-1} n)}~\dod~
\mathcal{S}_{\alpha}\quad \text{and}\quad
\frac{\sum_{k=1}^{\tau_n^{\ast}+1}\bar
\W_{\tau_k}-a(\mu^{-1}n)}{b(\mu^{-1} n)}~\dod~
\mathcal{S}_{\alpha},\quad n\to\infty.
\end{equation}

According to the central limit theorem for renewal processes
$$
\frac{\tau_n^{\ast}-\mu^{-1}n}{s\mu^{-3/2}\sqrt{n}}~\dod~
\mathcal{N}(0,1),\quad n\to\infty.
$$
This implies that, for $\varepsilon>0$ small enough, we can pick
$z=z(\varepsilon)$ so large that
$$
\P\{\tau_n^{\ast}\geq t_n\}\geq 1-\varepsilon,$$ where
$t_n:=[\mu^{-1}n-s\mu^{-3/2}z\sqrt{n}]$. Note that $n=\mu
t_n+O\big(t_n^{1/2}\big)$ and that
\begin{equation}\label{ineq}
\lim_{n\to\infty}\frac{a(t_n)-a\big(t_n+O\big(t_n^{1/2}\big)\big)}{b\big(t_n+O\big(t_n^{1/2}\big)\big)}=0\quad\text{and}\quad
\lim_{n\to\infty}\frac{b\big(t_n+O\big(t_n^{1/2}\big)\big)}{b(t_n)}=1.
\end{equation}
These can be easily checked with the exception of the case
$\alpha=1$ in which a proof of the first relation is needed: for
any $r\in (1,2]$,
\begin{eqnarray}
&&\frac{a\big(t_n+O\big(t_n^{1/r}\big)\big)-a(t_n)}{b(t_n)}\notag
\\&=&\frac{t_n \int_{Ct_n}^{Ct_n+O(t_n^{1/r})}\P\{\bar
\W_{\tau_1}>x\}{\rm d}x+
O(t_n^{1/r})\int_0^{Ct_n+O(t_n^{1/r})}\P\{\bar \W_{\tau_1}>x
\}{\rm d}x}{Ct_n}\notag \\&\leq& \frac{O\big(t_n^{1/r}\big)\log
t_n}{t_n}=o(1),\quad n\to\infty.\label{ineq3}
\end{eqnarray}
Motivated by our later needs we have proved this in a slightly
extended form with $r$ instead of $2$.

To prove the first relation in \eqref{eq:weak_limit_thm1_proof2}
we write, for $x\in \R$,
\begin{align*}
\P\left\{\frac{\sum_{k=1}^{\tau_n^{\ast}}\bar \W_{\tau_k}-a(\mu^{-1}n)}{b(\mu^{-1} n)}\leq x\right\}&\leq \varepsilon +
\P\left\{\frac{\sum_{k=1}^{t_n}\bar \W_{\tau_k}-a(\mu^{-1}n)}{b(\mu^{-1} n)}\leq x\right\}\\
&=\varepsilon + \P\left\{\frac{\sum_{k=1}^{t_n}\bar
\W_{\tau_k}-a\big(t_n+O\big(t_n^{1/2}\big)\big)}{b\big(t_n+O\big(t_n^{1/2}\big)\big)}\leq
x\right\}.
\end{align*}
Sending $n\to\infty$ in the last inequality and using
\eqref{eq:weak_limit_thm1_proof1} and \eqref{ineq} we obtain
$$\limsup_{n\to\infty}\P\left\{\frac{\sum_{k=1}^{\tau_n^{\ast}}\bar
\W_{\tau_k}-a(\mu^{-1}n)}{b(\mu^{-1} n)}\leq x\right\}\leq
\varepsilon+ \P\{\mathcal{S}_{\alpha}\leq x\}.$$ Letting now
$\varepsilon\to 0+$ yields
$$\limsup_{n\to\infty}\P\left\{\frac{\sum_{k=1}^{\tau_n^{\ast}}\bar
\W_{\tau_k}-a(\mu^{-1}n)}{b(\mu^{-1} n)}\leq x\right\}\leq
\P\{\mathcal{S}_{\alpha}\leq x\}.$$ A symmetric argument leads to
$$\liminf_{n\to\infty}\P\left\{\frac{\sum_{k=1}^{\tau_n^{\ast}}\bar
\W_{\tau_k}-a(\mu^{-1}n)}{b(\mu^{-1} n)}\leq x\right\}\geq
\P\{\mathcal{S}_{\alpha}\leq x\}.$$ The second relation in
\eqref{eq:weak_limit_thm1_proof2} follows in a similar manner.

\noindent {\sc Step 2. Limit theorems for $T_{S_n}$.}

\noindent {\sc Case $\alpha>1$}. Since $\E\xi^2<\infty$ and
$\sqrt{n}=o(b(\mu^{-1}n))$ we infer
$$\frac{S_n-(\E\xi)n}{b(\mu^{-1}n)}~\topr~ 0,\quad n\to\infty$$
by the central limit theorem. Now
\begin{equation}\label{alpha>1}
\frac{T_{S_n}-(\me\xi+2\mu^{-1}\me\bar\W_{\tau_1})n}{b(\mu^{-1}n)}~\dod~2\mathcal{S}_\alpha,\quad
n\to\infty
\end{equation}
follows from \eqref{eq:weak_limit_thm1_proof21} and \eqref{repr}
written in an equivalent form
$$T_{S_n}~\od~ (S_n-(\E\xi)n)+(\E\xi)n+2W_{S_n}+O_\P(1),\quad
n\to\infty.$$

\noindent {\sc Case $\alpha=1$}. Using the weak law of large
numbers and \eqref{eq:weak_limit_thm1_proof21} we arrive at
\begin{equation}\label{alpha1}
\frac{T_{S_n}-2a(\mu^{-1}n)}{C\mu^{-1} n}~\dod~ \frac{\mu
\E\xi}{C}+2\mathcal{S}_1,\quad n\to\infty.
\end{equation}

\noindent {\sc Case $\alpha<1$}. Since $n=o(b(\mu^{-1}n))$ we
conclude that $\frac{S_n}{b(\mu^{-1}n)}\topr 0$ as $n\to\infty$ by
the weak law of large numbers. This in combination with
\eqref{eq:weak_limit_thm1_proof21} and \eqref{repr} proves
\begin{equation}\label{alpha<1}
\frac{T_{S_n}}{(C \mu^{-1}n)^{1/\alpha}}~\dod~
2\mathcal{S}_\alpha,\quad n\to\infty.
\end{equation}

\noindent {\sc Step 3. Limit theorem for $T_n$}. Set
$$\nu(t)=\inf\{k\in\N: S_k>t\},\quad t\geq 0,$$ so that
$(\nu(t))_{t\geq 0}$ is the first passage time process associated
with the random walk $(S_k)_{k\in\N_0}$. The reason for
introducing $\nu(t)$ is justified by
\begin{equation}\label{ineq4}
T_{S_{\nu(n)-1}}\leq T_n\leq T_{S_{\nu(n)}},\quad n\in\N.
\end{equation}

\noindent {\sc Case $\alpha\geq 1$}. Fix any $r\in (1,2)$. Then
$\E\xi^r<\infty$ and thereupon
\begin{equation}\label{marc}
\nu(t)-(\E\xi)^{-1}t=o(t^{1/r}),\quad t\to\infty\quad \text{a.s.}
\end{equation}
by Theorem 4.4 on p.~89 in \cite{Gut:2009}.

\noindent {\sc Subcase $\alpha=1$}. Using \eqref{ineq4} we obtain,
for any $x\in\R$ and $\varepsilon>0$,
\begin{eqnarray*}
&&\P\Big\{\frac{T_n-2a((\mu\me\xi)^{-1}n)}{C(\mu\E\xi)^{-1}n}>x\Big\}\leq
\P\Big\{\frac
{T_{S_{\nu(n)}}-2a((\mu\me\xi)^{-1}n)}{C(\mu\E\xi)^{-1}n}>x\Big\}\\&\leq&
\P\big\{\nu(n)> (\me \xi)^{-1}n+\varepsilon
n^{1/r}\big\}+\P\Big\{\frac{T_{S_{[(\me \xi)^{-1}n+\varepsilon
n^{1/r}]}}-2a([(\me \xi)^{-1}n+\varepsilon
n^{1/r}])}{C(\mu\me\xi)^{-1} n}\\&+&\frac{2a([(\me
\xi)^{-1}n+\varepsilon
n^{1/r}])-2a((\mu\me\xi)^{-1}n)}{C(\mu\E\xi)^{-1}n}>x\Big\}.
\end{eqnarray*}
Letting $n\to\infty$ yields, for $x\in\R$,
$$\limsup_{n\to\infty}
\P\Big\{\frac{T_n-2a((\mu\me\xi)^{-1}n)}{C(\mu\E\xi)^{-1}n}>x\Big\}\leq
\P\Big\{\frac{\mu\me\xi}{C}+2\mathcal{S}_1>x\Big\}$$ having
utilized \eqref{ineq3}, \eqref{alpha1} and \eqref{marc}. Arguing
similarly we get the converse inequality for the lower limit,
thereby proving that
\begin{equation}\label{alp=1}
\frac{T_n-2a((\mu\me\xi)^{-1}n)}{C(\mu\E\xi)^{-1}n}~\dod~\frac{\mu\me\xi}{C}+2\mathcal{S}_1,\quad
n\to\infty.
\end{equation}

\noindent {\sc Subcase $\alpha>1$}. An analogous but simpler
argument enables us to show that \eqref{alpha>1} entails
\begin{equation}\label{alp>1}
\frac{T_n-(1+2(\mu\E\xi)^{-1}\me\bar\W_{\tau_1})n}{b((\mu\E\xi)^{-1}n)}~\dod~2\mathcal{S}_\alpha,\quad
n\to\infty.
\end{equation}

\noindent {\sc Case $\alpha<1$}. The proof given for the case
$\alpha\geq 1$ does not work in the case ${\rm (A1)}$ when {\ma
$\alpha\leq 1/2$} because it is then not necessarily true that
$\me \xi^r<\infty$ for some $r>1$. In view of this we use the weak
law of large numbers
\begin{equation}\label{weak}
\frac{\nu(t)}{t}~\topr~ \frac{1}{\mu},\quad t\to\infty
\end{equation}
rather than the Marcinkiewicz-Zygmund strong law \eqref{marc}.

Another appeal to \eqref{ineq4} gives, for any $x\in\R$ and
$\varepsilon>0$,
\begin{multline*}
\P\left\{\frac{T_n}{(C(\mu\E\xi)^{-1}n)^{1/\alpha}}>x\right\}\leq \P\left\{\frac{T_{S_{\nu(n)}}}{(C(\mu\E\xi)^{-1}n)^{1/\alpha}}>x\right\}\\
\leq
\P\{\nu(n)>((\E\xi)^{-1}+\varepsilon)n\}+\P\left\{\frac{T_{S_{[((\E\xi)^{-1}+\varepsilon)n]}}}{(C(\mu\E\xi)^{-1}n)^{1/\alpha}}>x\right\}.
\end{multline*}
Sending $n\to\infty$ we obtain with the help of \eqref{alpha<1}
and \eqref{weak}
$$\limsup_{n\to\infty}\P\left\{\frac{T_n}{(C(\mu\E\xi)^{-1}n)^{1/\alpha}}>x\right\}\leq
\P\{2\mathcal{S}_{\alpha}>x(1+\varepsilon \E\xi)^{-1/\alpha}\}.$$
Letting $\varepsilon\to 0+$ and using continuity of the
distribution of $\mathcal{S}_{\alpha}$ yields
$$\limsup_{n\to\infty}\P\left\{\frac{T_n}{(C(\mu\E\xi)^{-1}n)^{1/\alpha}}>x\right\}\leq
\P\{2\mathcal{S}_{\alpha}>x\}.$$ The converse inequality for the lower limit can be derived analogously. Thus,
\begin{equation}\label{alp<1}
\frac{T_n}{(C(\mu\E\xi)^{-1}n)^{1/\alpha}}~\dod~2\mathcal{S}_{\alpha},\quad
n\to\infty.
\end{equation}
The proof of Theorem \ref{thm:main1T} is complete.

\begin{proof}[Proof of Corollary \ref{cor:main1X}] The forms of
limit relations for $T_n$ in our Theorem \ref{thm:main1T} and
Theorem on pp.~146--148 in \cite{kesten1975limit} are the same,
only the values of constants differ. In view of this the limit
relations for $X_k$ in our setting are obtained by copying the
corresponding limit relations from the aforementioned theorem in
\cite{kesten1975limit}.
\end{proof}

\subsection{Proof of Theorem \ref{thm:main2T} and Corollary \ref{cor:main2X}}\label{7.3}

The proof goes the same path as that of Theorem \ref{thm:main1T}.
However, appearance of nontrivial slowly varying factors leads to
minor technical complications. We shall only give the weak
convergence results explicitly (recall that in the formulation of
Theorem \ref{thm:main2T} normalizing and centering functions were
not specified). Also, we shall check several claims wherever we
feel it is necessary.

According to Proposition \ref{prop:tail_main} (parts ${\rm (C3)}$
and ${\rm (C4)}$), $$\P\{\bar \W_{\tau_1}>x\}~\sim~ \E\tau_1 \E
\vartheta^\alpha x^{-\alpha}\ell(x^{1/2}),\quad x\to\infty,$$
where $\alpha=\beta/2$ in case $(B2)$. Therefore, limit relation
\eqref{eq:weak_limit_thm1_proof1} holds with some $a(t)$ and
$b(t)$. To identify them we need more notation. For $\alpha\in
(1/2, 2)$, let $c_\alpha(t)$ be any positive function satisfying
$\lim_{t\to\infty}t\P\{\bar \W_{\tau_1}>c_\alpha(t)\}=1$. Further,
assuming that $\alpha=2$ let $r_2(t)$ be any positive function
satisfying $\lim_{t\to\infty}\int_{[0,\,r_2(t)]} x^2 {\rm
d}\P\{\bar \W_{\tau_1}\leq x\}/(r_2(t))^2=1$. By Lemma 6.1.3 in
\cite{iksanovrenewal}, $c_\alpha(t)$ and $r_2(t)$ are regularly
varying at $\infty$ of indices $1/\alpha$ and $1/2$, respectively.
For the latter, the fact is also needed that the function
$t\mapsto \int_{[0,\,r_2(t)]} x^2 {\rm d}\P\{\bar \W_{\tau_1}\leq
x\}$ is slowly varying at $\infty$. Observe that the case
$\alpha=2$ only arises under the assumptions ${\rm (B1)}$ which
then ensure that $\E\xi^2=\infty$. This in combination with the
aforementioned lemma yields
\begin{equation}\label{r2t}
\lim_{t\to\infty}t^{-1/2}r_2(t)=\infty.
\end{equation}

Using again Theorem 3 on p.~580 and formula (8.15) on p.~315 in \cite{feller:1971} we obtain

\noindent $b(t)=c_\alpha(t)$ and $a(t)=0$ if $\alpha\in
(1/2,1)$;

\noindent $b(t)=c_1(t)$ and $a(t)=t\int_0^{c_1(t)}\P\{\bar \W_{\tau_1}>x\}{\rm
d}x$ if $\alpha=1$;

\noindent $b(t)=c_\alpha(t)$ and $a(t)=(\E \bar \W_{\tau_1})t$ if
$\alpha\in (1,2)$;

\noindent $b(t)=r_2(t)$ and $a(t)=(\E \bar \W_{\tau_1})t$ if
$\alpha=2$.

\noindent {\sc Case $\alpha\in (1/2, 1)$}. Repeating verbatim the proof of Theorem \ref{thm:main1T} we obtain
\begin{equation}\label{IIalp<1}
\frac{T_n}{(\mu\E\xi)^{-1/\alpha}c_\alpha(n)}~\dod~2\mathcal{S}_\alpha,\quad n\to\infty.
\end{equation}

\noindent {\sc Case $\alpha=1$}. We need an analogue of relation \eqref{ineq3}: for $r\in (1,2]$, as $n\to\infty$,
\begin{align*}
&\frac{a\big(t_n+O\big(t_n^{1/r}\big)\big)-a(t_n)}{b(t_n)}\\
&=\frac{t_n \int_{c_1(t_n)}^{c_1(t_n+O(t_n^{1/r}))}\P\{\bar
\W_{\tau_1}>x\}{\rm
d}x+O(t_n^{1/r})\int_0^{c_1(t_n+O(t_n^{1/r}))}\P\{\bar
\W_{\tau_1}>x \}{\rm d}x}{c_1(t_n)} \\
&\leq\frac{t_n
\P\{\bar
\W_{\tau_1}>c_1(t_n)\}(c_1(t_n+O(t_n^{1/r}))-c_1(t_n))}{c_1(t_n)}+
\frac{O(t_n^{1/r})\int_0^{c_1(t_n+O(t_n^{1/r}))}\P\{\bar
\W_{\tau_1}>x\}{\rm d}x}{c_1(t_n)}=o(1).
\end{align*}
The first summand tends to zero in view of two facts:
$\lim_{n\to\infty} t_n \P\{\bar \W_{\tau_1}>c_1(t_n)\}=1$ by the
definition of $c_1(t)$ and $\lim_{n\to\infty}
\big(c_1\big(t_n+O\big(t_n^{1/r}\big)\big)-c_1(t_n)\big)/
c_1(t_n)=0$ which is a consequence of regular variation of
$c_1(t)$. The second summand tends to zero because
$\int_0^{c_1(t)}\P\{\bar \W_{\tau_1}>x\}{\rm d}x$ is slowly
varying at $\infty$ as a superposition of the slowly varying and
regularly varying functions.

For Step 2 we need the following modified argument. In view of
$(\xi2)$ the function $\P\{\xi>t\}$ is regularly varying at
$\infty$ of index $-2$ and $\E\xi^2$ can be finite or infinite.
Therefore, $S_n$ satisfies the central limit theorem with
normalization sequence which is regularly varying at $\infty$ of
index $1/2$. Since $c_1(t)$ is regularly varying at $\infty$ of
order $1$ we infer $$\frac{S_n-(\E\xi)n}{c_1(n)}~\topr~0,\quad
n\to\infty$$ and thereupon
\begin{equation*}
\frac{T_{S_n}-(\me\xi)n-2a(\mu^{-1}n)}{\mu^{-1}c_1(n)}~\dod~2\mathcal{S}_1,\quad
n\to\infty.
\end{equation*}
To pass from this limit relation to the final result
\begin{equation}\label{IIalp=1}
\frac{T_n-n-2a((\mu\E\xi)^{-1}n)}{(\mu\E\xi)^{-1}c_1(n)}~\dod~2\mathcal{S}_1,\quad n\to\infty,
\end{equation}
that is, to realize Step 3, one can mimic the proof of Theorem \ref{thm:main1T}.

\noindent {\sc Case $\alpha\in (1,2]$}. While implementing Step 2
in the case $\alpha=2$ one uses the fact that according to
\eqref{r2t} $b(t)=r_2(t)$ satisfies $\sqrt{n}=o(r_2(\mu^{-1}n))$
as $n\to\infty$. Since the other parts of the proof of Theorem
\ref{thm:main1T} do not require essential changes we arrive at
\begin{equation}\label{IIalp>1}
\frac{T_n-(1+2(\mu\E\xi)^{-1}\me\bar\W_{\tau_1})n}{(\mu\E\xi)^{-1/\alpha}c_\alpha(n)}~\dod~2\mathcal{S}_\alpha,\quad
n\to\infty,
\end{equation}
when $\alpha\in (1,2)$, and
\begin{equation}\label{IIalp=2}
\frac{T_n-(1+2(\mu\E\xi)^{-1}\me\bar\W_{\tau_1})n}{(\mu\E\xi)^{-1/2}r_2(n)}~\dod~2\mathcal{N}(0,1),\quad
n\to\infty,
\end{equation}
when $\alpha=2$. The proof of Theorem \ref{thm:main2T} is complete.

\begin{proof}[Proof of Corollary \ref{cor:main2X}]

Since $(T_n)_{n\in\N_0}$ is an `inverse' sequence for
$(X_k)_{k\in\N_0}$ we can use a standard inversion technique (see,
for instance, the proof of Theorem 7 in \cite{feller1949}) to pass
from the distributional convergence of $T_n$, properly centered
and normalized, as $n\to\infty$ to that of $X_k$, again properly
centered and normalized, as $k\to\infty$. Additional complications
arising in the case $\alpha=1$ can be handled with the help of
arguments given in Section 3 of \cite{Anderson+Athreya:1988}.

Here are the limit relations for $X_k$, properly normalized and
centered, as $k\to\infty$ which correspond to \eqref{IIalp<1},
\eqref{IIalp=1}, \eqref{IIalp>1} and \eqref{IIalp=2}:

\noindent if $\alpha\in (1/2,1)$, then
\begin{equation}\label{IIIalp<1}
\P\{\bar \W_{\tau_1}>k\} X_k~\dod~ \mu \E\xi
(2S_\alpha)^{-\alpha};
\end{equation}

\noindent if $\alpha=1$, then
\begin{equation}\label{IIIalp=1}
\frac{X_k-s(k)}{t(k)}~\dod~-\mathcal{S}_1,
\end{equation}
where, with $m(t):=\int_0^t\P\{\bar \W_{\tau_1}>x\}{\rm d}x$ for
$t>0$ and $b:=(\mu\E\xi)^{-1}$,
$$s(k):=\frac{k}{1+2bm(c_1(bk/(1+2bm(bk))))},\quad
k\in\N$$ and
$$t(k):=\frac{c_1(k/m(k))}{1+2bm(k)},\quad k\in\N$$ (we do not write $2b m(k)$ instead of $1+2b m(k)$ because the case $\lim_{t\to\infty} m(t)=
\E\bar \W_{\tau_1}<\infty$ is not excluded);

\noindent if $\alpha\in (1,2)$, then
\begin{equation}\label{IIIalp>1}
\frac{X_k-(1+2(\mu\E\xi)^{-1}\me\bar\W_{\tau_1})^{-1}k}{c_\alpha(k)}~\dod~
-2(\mu\E\xi)^{-1/\alpha}(1+2(\mu\E\xi)^{-1}\me\bar\W_{\tau_1})^{-(1+1/\alpha)}\mathcal{S}_\alpha;
\end{equation}

\noindent if $\alpha=2$, then
\begin{equation}\label{IIIalp=2}
\frac{X_k-(1+2(\mu\E\xi)^{-1}\me\bar\W_{\tau_1})^{-1}k}{r_2(k)}~\dod~
2(\mu\E\xi)^{-1/2}(1+2(\mu\E\xi)^{-1}\me\bar\W_{\tau_1})^{-3/2}\mathcal{N}(0,1).
\end{equation}
The proof of Corollary \ref{cor:main2X} is complete.
\end{proof}

\subsection{Proof of auxiliary Lemmas \ref{lem:zmom}, \ref{lem:Wnu}, \ref{prop:main2} and \ref{lem:1}}\label{sec:proofs_lemmas}
\subsubsection{Proof of Lemma \ref{lem:zmom}}
\begin{proof}[Proof of Lemma \ref{lem:zmom}]
To prove \eqref{eq:5} we first represent $Z_{S_n-1}$ as a sum of
independent random variables
\begin{equation}\label{decomp}
Z_{S_n-1} = \sum_{j=1}^{\Z_{n-1}}V_j^{(n)} + \wt V^{(n)},\quad
n\in\N\quad \text{a.s.},
\end{equation}
where $V_j^{(n)}$ is the number of progeny residing in the
generation $S_n-1$ of the $j$th particle in the generation
$S_{n-1}$ and $\wt V^{(n)}$ is the number of progeny residing in
the generation $S_n-1$ of the immigrants arriving in the
generations $S_{n-1},\ldots, S_n-2$. For later use, we note that,
under $\P_\omega$,
\begin{equation}\label{distrequ}
V_j^{(n)}~\od~ \Zc(1,\xi_n-1)\quad \text{and}\quad \wt
V^{(n)}~\od~ \Zc_{\xi_n-1},\quad n\in\N,
\end{equation}
where $\omega$ is assumed independent of $(\Zc_k)_{k\in\N_0}$ a
Galton--Watson process with unit immigration and ${\rm Geom}(1/2)$
offspring distribution.

With the help of \eqref{decomp} we now write a standard
decomposition for the number of particles in the generation $S_n$
over the particles comprising the generation $S_{n-1}$ and their
offspring
\begin{equation}\label{eq:w1}
\Z_n =\sum_{j=1}^{\Z_{n-1}}\sum_{i=1}^{V_j^{(n)}} U^{(n)}_{i,j} +
\sum_{i=1}^{\wt V^{(n)}}\wt U_i^{(n)} +
U_0^{(n)}=:\sum_{j=1}^{\Z_{n-1}}\V_j^{(n)} + \wt \V^{(n)} +
U_0^{(n)},\quad n\in\N\quad \text{a.s.}
\end{equation}
Here, the notation $U_{i,j}^{(n)}$, $\wt U_i^{(n)}$, $U_0^{(n)}$
is self-explained. For instance, $U_0^{(n)}$ is the number of
offspring of the immigrant arriving in the generation $S_n-1$.
Observe that, under $\P_\omega$, $(U_{i,j}^{(n)})_{i,j\in \N}$,
$(\wt U_i^{(n)})_{i\in\N}$ and $U_0^{(n)}$ are independent with
distribution ${\rm Geom} (\lambda_n)$. In what follows, for
simplicity we omit the superscripts $(n)$: for instance, we write
$\V_j$ for $\V_j^{(n)}$ and similarly for the other variables. The
following formulas play an important role in the subsequent proof:
\begin{align}
\E_\omega [ U_0 | \Z_{n-1}] =\E_\omega U_0  =\rho_n,& \quad \E_\omega [U_0^2 | \Z_{n-1} | = \E_\omega U_0^2 = 2\rho_n^2 + \rho_n\notag \\
\E_\omega\big[ \V_i \big|  \Z_{n-1} \big] = \E V_i \cdot \rho_n =
\rho_n,&\quad \E_\omega\big[ \wt \V \big|  \Z_{n-1} \big] =
(\xi_n-1)\rho_n.\label{form}
\end{align}
The two cases $\kappa\in (0,1]$ and $\kappa \in (1,2]$ should be
treated separately.
\medskip

\noindent {\sc Case $\kappa \leq 1$}. By Jensen's inequality and
subadditivity of the function $s \mapsto s^\kappa$ on $[0,\infty)$
\begin{multline*}
\E_\omega[ \Z_n^\kappa | \Z_{n-1}]  \leq  \left(\E_\omega[ \Z_n |
\Z_{n-1}]\right)^\kappa = \left[\E_\omega \Big[
\sum_{j=1}^{\Z_{n-1}}\V_j + \wt \V + U_0 \Big| \Z_{n-1}
\Big]\right]^\kappa \\ \leq \Big(\Z_{n-1}\rho_n + (\xi_n-1)\rho_n
+ \rho_n\Big)^\kappa \leq \Z_{n-1}^\kappa\rho_n^\kappa +
\xi_n^\kappa\rho_n^\kappa.
\end{multline*}
Taking the expectations we obtain
\begin{equation*}
\E \Z_n^\kappa \le \gamma \E \Z_{n-1}^\kappa + \E(\rho\xi)^\kappa
\end{equation*}
which entails \eqref{eq:5}.

\noindent {\sc Case $\kappa \in (1,2]$}. An application of
conditional Jensen's inequality yields
\begin{equation} \label{eq:w7}
\E_\omega \Z_n^\kappa  = \E_\omega \Big[   \E_\omega\big[
\Z_n^\kappa \big|  \Z_{n-1} \big] \Big] \le  \E_\omega \Big[
\Big( \E_\omega\big[ \Z_n^2 \big|  \Z_{n-1} \big]\Big)^{\kappa/2}
\Big].
\end{equation}
To estimate the conditional second moment we represent it as
follows
\begin{align*}
\E_\omega\big[ \Z_n^2 \big|  \Z_{n-1} \big] & =  \E_\omega\bigg[
\bigg(\sum_{j=1}^{\Z_{n-1}} \V_j + \wt \V + U_0  \bigg)^2 \bigg|
\Z_{n-1} \bigg]
= \sum_{1\le i\not= j \le \Z_{n-1}} \E_\omega\big[ \V_i \big| \Z_{n-1} \big] \E_\omega\big[ \V_j \big| \Z_{n-1} \big]\\
& + \sum_{j=1}^{\Z_{n-1}} \E_\omega\big[ \V_j^2 \big|  \Z_{n-1} \big]+ 2 \E_\omega[ \widetilde\V +U_0| \Z_{n-1}]
\E_\omega\bigg[\sum_{j=1}^{\Z_{n-1}} \V_j \bigg|  \Z_{n-1} \bigg] +    \E_\omega\big[ \wt\V^2 \big|  \Z_{n-1} \big]\\
&+    \E_\omega\big[ U_0^2 \big| \Z_{n-1} \big] + 2
\E_\omega\big[ \wt\V \big| \Z_{n-1} \big]    \E_\omega\big[ U_0
\big|  \Z_{n-1} \big].
\end{align*}
Appealing now to \eqref{form} we conclude that
\begin{equation}\label{eq:aux1}
\E_\omega\big[ \Z_n^2 \big|  \Z_{n-1} \big] \le  \Z_{n-1}^2
\rho_n^2 + \Z_{n-1} \E_{\omega} \V_1^2 + 2\Z_{n-1} \xi_n
\rho_n^2+\E_{\omega} \wt \V^2 + 2\rho_n^2 + \rho_n + 2
\xi_n\rho_n^2.
\end{equation}
Plugging the last inequality into \eqref{eq:w7} and using
subadditivity once again we obtain
\begin{multline}\label{eq:lem_mom_proof1}
\E\Z_n^\kappa\le\gamma \E
\Z_{n-1}^\kappa+\bigg(\E\Z_{n-1}^{\kappa/2}\cdot \E \Big[ \big(
\E_{\omega} \V_1^2 \big)^{\kappa/2} \Big] + 2\E
\Z_{n-1}^{\kappa/2}\cdot \E \xi^{\kappa/2}\rho^\kappa\\+ \E\Big[
\big(\E_\omega\wt\V^2 \big)^{\kappa/2}
\Big]+2\gamma+\E\rho^{\kappa/2} + 2 \E
\xi^{\kappa/2}\rho^\kappa\bigg).
\end{multline}
Next, we check that
\begin{equation}\label{eq:w8} \E \Big[ \big(
\E_{\omega} \V_1^2 \big)^{\kappa/2} \Big] <\8 \quad \mbox{ and
}\quad \E \Big[ \big(\E_\omega \wt\V^2 \big)^{\kappa/2} \Big]< \8.
\end{equation}
With the help of $$\E_\omega V_i=1 \quad \text{and}\quad {\rm
Var}_{\omega}V_i= 2(\xi_n-1)$$ which is a consequence of
\eqref{distrequ} and \eqref{eq:harris} we infer
\begin{align*}
\E \Big[ \big(\E_\omega \V_1^2 \big)^{\kappa/2} \Big] & = \E \bigg[ \bigg( \E_\omega
\bigg( \sum_{j=1}^{V_1} U_{1,j} \bigg)^2\bigg)^{\kappa/2} \bigg] = \E \bigg[ \bigg( \E_\omega
\bigg[  \sum_{1\le j\not=l \le V_1} U_{1,j} U_{1,l} + \sum_{j=1}^{V_1} U_{1,j}^2  \bigg]\bigg)^{\kappa/2} \bigg]\\
& \le \E \Big(\rho_n^2\E_{\omega} V_1^2 + (2
\rho_n^2+\rho_n)\E_\omega V_1\Big)^{\kappa/2}\le 2^{\kappa/2} \E
\xi^{\kappa/2}\rho^\kappa +\gamma+ \E \rho^{\kappa/2} <\8.
\end{align*}
A similar argument in combination with $\E_{\omega}\wt V =
\xi_n-1$ leads to the conclusion
\begin{equation*}
\E \Big[ \big( \E_{\omega} \wt\V^2 \big)^{\kappa/2} \Big]  =
\E\Big(\rho_n^2\E_{\omega} \wt V^2+
(\rho_n^2+\rho_n)\E_{\omega}\wt V \Big)^{\kappa/2}\le \E \Big[
\big( \rho_n^2 \E_{\omega} \wt V^2 \big)^{\kappa/2} \Big]+ \E
\xi^{\kappa/2}\rho^k+\me (\rho\xi)^{\kappa/2}.
\end{equation*}
Left with the proof of finiteness of the first term on the
right-hand side we represent $\wt V$ as a sum of independent
random variables $$\wt V=\wt V^{(n)} = \sum_{i=1}^{\xi_n-1} \wt
V^{(n)}_i,\quad n\in\N\quad \text{a.s.},$$ where, for $1\leq i\leq
\xi_n-1$, $\wt V^{(n)}_i$ is the number of progeny residing in the
generation $S_n-1$ of the immigrant arriving in the generation
$S_n-i$. Under $\P_\omega$, $\wt V_i^{(n)} \od \Zc(i,\xi_n-1)$,
where $\omega$ is assumed independent of $(\Zc(i,k))_{k\geq i}$.
With this at hand, an appeal to \eqref{eq:harris} yields
$$\E_{\omega}\wt V_i^2=\E_\omega [\Zc(i,\xi_n-1)]^2=\E_\omega
[\Zc(1,\xi_n-i)]^2 =2(\xi_n-i)+1\leq 2\xi_n$$ and $\E_{\omega}\wt
V_i=1$. Here and hereafter, to ease the notation we write $\wt
V_i$ for $\wt V_i^{(n)}$. Finally,
\begin{eqnarray*}
\E \Big[ \big(\rho_n^2 \E_{\omega} \wt V^2 \big)^{\kappa/2} \Big]
&=& \E\rho_n^\kappa \bigg(\E_{\omega} \bigg( \sum_{i=1}^{\xi_n-1}
\wt V_i \bigg)^2 \bigg)^{\kappa/2}=  \E\rho_n^\kappa
\bigg(\sum_{i=1}^{\xi_n-1} \E_{\omega} \wt V_i^2 + \sum_{1\le
i\not= j < \xi_n} \E_{\omega} \wt V_i \E_{\omega}\wt V_j
\bigg)^{\kappa/2}\\&\le& (5/2)^{\kappa/2}\E (\rho\xi)^\kappa <\8
\end{eqnarray*}
which finishes the proof of \eqref{eq:w8}.

Turning to the asymptotic behavior of $\E\Z_{n-1}^{\kappa/2}$
which appears on the right-hand side of \eqref{eq:lem_mom_proof1}
we consider yet another two cases.

\noindent {\sc Case $\gamma \leq 1$} in which $\E \rho^{\kappa/2}
< 1$. To see it, observe that when $\gamma=1$ the inequality
$\E\rho^{\kappa/2}<\gamma^{1/2}$ is strict because the assumption
$\E\log \rho\in [-\infty, 0)$ implies that the distribution of
$\rho$ is nondegenerate at $1$. By the already proved inequality
\eqref{eq:5} for powers $\leq 1$
$$\sup_n \E\Z_n^{\kappa/2}<\infty$$ which in combination with \eqref{eq:w8} shows that the expression
in the parentheses in \eqref{eq:lem_mom_proof1} is bounded. This
ensures \eqref{eq:5}.

\noindent {\sc Case $\gamma>1$}. By the already proved inequality
\eqref{eq:5} for powers $\leq 1$
$$\E \Z_n^{\kappa/2}\leq Ca_n,\quad n\in\N,$$
where $a_n=1$ or $=n$ or $=[\E \rho^{\kappa/2}]^n$ depending on
whether $\E \rho^{\kappa/2}<1$ or $\E \rho^{\kappa/2}=1$ or $\E
\rho^{\kappa/2}>1$. Since in any event $a_n\leq \gamma^{n/2}$ for
$n\in\N$, \eqref{eq:lem_mom_proof1} entails $$ \E
\Z_n^{\kappa}\leq \gamma\E \Z_{n-1}^{\kappa} +
C_1\gamma^{n/2},\quad n\in\N$$ for some $C_1>0$. Iterating this
yields $\E \Z_n^{\kappa}\leq C_2\gamma^n$ for some $C_2>1$ and all
$n\in\N$, thereby finishing the proof of \eqref{eq:5} in the case
$\gamma> 1$ and in general.

To prove \eqref{eq:lu1} we use a decomposition $\W_1 = W_{\xi_1-1}
+ \Z_1$ a.s. Inequality \eqref{eq:5} tells that we are left with
checking that $$\E W_{\xi_1-1}^\kappa <\8.$$ Since, under
$\P_\omega$, $W_{\xi_1-1}\od \Wc_{\xi_1-1}$, where $\omega$ is
assumed independent of $(\Wc_n)_{n\in\N_0}$, an application of
Lemma \ref{lem:exp_moments} yields
\begin{equation*}
\E[W^\kappa_{\xi_1-1}] = \sum_{j\geq 0} \E\big[ {\bf
1}_{\{\xi_1=j+1\}} (\Wc_j)^\kappa \big]\le C\sum_{j\geq 0}
\P\{\xi=j+1\} j^{2\kappa}= C \E (\xi-1)^{2\kappa} <\8
\end{equation*}
for a positive constant $C$. The proof of Lemma \ref{lem:zmom} is
complete.
\end{proof}

\subsubsection{Proof of Lemma \ref{lem:Wnu}}

\begin{proof}[Proof of Lemma \ref{lem:Wnu}] We start by proving \eqref{5??}. Pick
$\kappa_0\in (0,\kappa)$, put $p=\kappa/\kappa_0$ and choose $q$
such that $1/p+1/q=1$. According to Lemma \ref{lem:zmom},
\begin{equation}\label{inter2}
\E \Z_n^\kappa\leq C,\quad n\in\N
\end{equation}
for a positive constant $C$, whence
$$\E \bigg(\sum_{i=1}^n \Z_i\bigg)^\kappa \leq C\max(n^\kappa, n),\quad n\in\N$$ by subadditivity (convexity)
of $x\mapsto x^\kappa$ when $\kappa\in (0,1]$ ($\kappa\in (1,2]$).
By Lemma \ref{lem:nu}, $\P\{\tau_1=n\}\leq C_1e^{-C_2n}$ for all
$n\in\N$ and positive constants $C_1$ and $C_2$. With these at
hand, an application of H\"older's inequality yields
\begin{multline*}
\E\bigg[\bigg(\sum_{i=1}^{\tau_1}\Z_i\bigg)^{\kappa_0}\bigg]=
\sum_{n\geq 1}\E\bigg[\bigg(\sum_{i=1}^{\tau_1}
\Z_i\bigg)^{\kappa_0}{\bf 1}_{\{\tau_1=n\}}\bigg]\\
\le \sum_{n\geq 1} \bigg( \E \bigg(\sum_{i=1}^n \Z_i\bigg)^\kappa
\bigg)^{1/p} \cdot \P\{\tau_1=n\}^{1/q} \le C^{1/p}C_1 \sum_{n\geq
1}\max (n^{\kappa/p}, n^{1/p})e^{-C_2 n/q}<\infty.
\end{multline*}
The proof of \eqref{5??} is complete.

Turning to the proof of \eqref{eq:5?} we shall only show that
\begin{equation} \label{eq:w10}
\E\big( \W_n^\downarrow  \big)^\kappa\leq C,\quad n\geq 2
\end{equation}
for a positive constant $C$. Formula \eqref{eq:5?} then follows
with the help of the same argument (involving H\"older's
inequality) that we used while proving \eqref{5??}.

For $i\geq 2$ and $1\leq j\leq \Z_{i-1}=Z_{S_{i-1}}$, denote by
$U_j^{(i)}$ the number of progeny in the generations
$S_{i-1}+1,\ldots, S_i-1$ of the $j$th particle in the generation
$S_{i-1}$. Here is a representation of $\W_i^\downarrow$ which is
slightly different from the original definition
$$\W_i^\downarrow=\sum_{j=1}^{\Z_{i-1}}U^{(i)}_j,\quad i\geq 2.$$
Under $\P_\omega$, $U_j^{(i)}\od \Yc(1, \xi_i-1)$ for $i\geq 2$,
where we set $\Yc(1,0)=0$ and $\omega$ is assumed independent of
$(\Yc(1, k))_{k\in\N}$. In particular, according to
\eqref{eq:harris}
\begin{equation}\label{inter}
\E_\omega U^{(i)}_j = \xi_i-1\quad \text{and}\quad \E_\omega \big[U^{(i)}_j\big]^2\leq 3\xi_i^3.
\end{equation}

We shall treat the cases $\kappa\in (0,1]$ and $\kappa\in (1,2]$ separately.

\noindent {\sc Case $\kappa\in (0,1]$}. Under $\P_\omega$, for
$1\leq j\leq \Z_{i-1}$, $U_j^{(i)}$ is independent of $\Z_{i-1}$.
This in combination with \eqref{inter} proves that
$$\E_{\omega}[\W_i^\downarrow | \Z_{i-1} ]= \Z_{i-1}
(\xi_i-1),\quad i\geq 2.$$ Therefore, we obtain
$$\E\big( \W_i^\downarrow  \big)^\kappa  \leq \E\big[ [\E_\omega (\W_i^\downarrow | \Z_{i-1})]^\kappa\big]
\leq \E \xi^\kappa \E \Z_{i-1}^\kappa\leq C,\quad i\geq 2$$ having
utilized Jensen's inequality, \eqref{inter2} and the fact that
$\xi_i$ and $\Z_{i-1}$ are independent.

\noindent {\sc Case $\kappa\in (1,2]$}. Another application of
Jensen's inequality in combination with \eqref{inter},
\eqref{inter2} and subadditivity of $x\mapsto x^{\kappa/2}$ on
$[0,\infty)$ yields, for $i\geq 2$,
\begin{align*}
\E\big( \W_i^\downarrow \big)^\kappa & = \E\bigg[
\E_\omega\bigg[\bigg(\sum_{j=1}^{\Z_{i-1}} U^{(i)}_j
\bigg)^\kappa\Big| \Z_{i-1} \bigg] \bigg] \le \E\bigg(
\E_\omega\bigg[\bigg(\sum_{j=1}^{\Z_{i-1}} U^{(i)}_j \bigg)^2
\Big| \Z_{i-1} \bigg] \bigg)^{\kappa/2}\\&= \E\bigg[\bigg(
\sum_{1\le l\not= j \le \Z_{i-1}} \E_{\omega} U^{(i)}_l\E_{\omega}
U^{(i)}_j + \sum_{j=1}^{\Z_{i-1}} \E_\omega \big[U^{(i)}_j\big]^2
\bigg)^{\kappa/2} \bigg]\le \E \Z_{i-1}^\kappa  \E \xi^\kappa +
3\E \Z_{i-1}^{\kappa/2} \E\xi^{3\kappa/2}\leq C
\end{align*}
for a positive constant $C$. The proof of \eqref{eq:w10} is
complete.
\end{proof}

\subsubsection{Proof of Lemma \ref{prop:main2}}

We follow the method invented by Kesten et al.
\cite{kesten1975limit}. While some parts of the proofs given in
\cite{kesten1975limit} can be directly transferred to our setting,
the others require an additional work. We do not present all the
details of the proof focussing instead on the main differences.

We begin with a brief overview of the arguments leading to the
claim of Lemma~\ref{prop:main2}. Given a large positive constant
$A$, put $$\sigma=\sigma(A):=\min \{i\in\N: Z_j >A  \mbox{ for
some } j\le S_i\}.$$ Thus, we observe the process
$(Z_n)_{n\in\N_0}$ up to the first time $j$ when it exceeds the
level $A$ and then put $\sigma=i$ for the smallest index $i$
satisfying $S_i\ge j$. The following decomposition holds
$$\sum_{k=1}^{\tau_1}\big( \Z_k +  \W_k^\downarrow \big)=\sum_{k=1}^{\tau_1}\big( \Z_k +  \W_k^\downarrow \big){\bf 1}_{\{\sigma\geq \tau_1\}}+
\Big(\sum_{k=1}^{\sigma-1} \big( \Z_k +  \W_k^\downarrow \big) +
\mathbb{S}_\sigma + \sum_{i=\sigma+1}^{\tau_1}
\Y^{\downarrow}_i\Big){\bf 1}_{\{\sigma<\tau_1\}}\quad
\text{a.s.},$$ where $\mathbb{S}_\sigma$ is the number of
particles in the generation $S_\sigma$ plus their total progeny,
and, for $i\in\N$, $\Y^\downarrow_i$ is the total progeny in the
generations $S_i+1, S_i+2,\ldots$ of the immigrants arriving in
the generations $S_{i-1},\ldots, S_i-1$.

We intend to prove that the first, second and fourth summands on
the right-hand side of this decomposition are negligible in a
sense to be made precise, so that $$\sum_{k=1}^{\tau_1} \big( \Z_k
+ \W_k^\downarrow \big) \approx \mathbb{S}_\sigma {\bf
1}_{\{\sigma<\tau_1\}}.$$ In view of the definition of
$\mathbb{S}_\sigma$ and the fact that $\Z_{\sigma}=Z_{S_{\sigma}}
\approx A$ for $A$ as above one can expect that $\mathbb{S}_\sigma
{\bf 1}_{\{\sigma<\tau_1\}} \approx \Z_\sigma \E_\omega
[Y(S_\sigma, \infty)]{\bf 1}_{\{\sigma<\tau_1\}}$. We shall
demonstrate that the variable $\E_{\omega}[Y(S_\sigma, \infty)]$
is related to a random difference equation whose tail behavior
determines that of $\mathbb{S}_\sigma$.

To realize the programme just outlined we need two auxiliary
results.
\begin{lem} \label{lem:p3}
Assume that the assumptions of Lemma \ref{prop:main2} hold. Then,
for any $A>0$, as $x\to\infty$,
\begin{equation}\label{eq:aux}
\P\bigg\{\sum_{k=1}^{\tau_1} \big( \Z_k +  \W_k^\downarrow \big)>
x, \: \sigma \ge \tau_1 \bigg\}+ \P\bigg\{
\sum_{k=1}^{\sigma-1}\big(\Z_k +  \W_k^\downarrow \big)> x, \:
\sigma < \tau_1 \bigg\}  = o(x^{-\a}).
\end{equation}
\end{lem}
\begin{proof}
We only give a proof for the first summand in \eqref{eq:aux}. The
second summand can be treated along similar lines.

The random variable $\tau_1$ has a finite exponential moment by
Lemma \ref{lem:nu}. Furthermore, $\tau_1$ does not depend on the
future of the sequence $(\xi_i)_{i\in\N}$. Therefore, the
assumption $\E\xi^{3\alpha/2}<\infty$ ensures that
\begin{equation}\label{eq:mom1}
\E [S_{\tau_1}]^{3\alpha/2}<\infty
\end{equation}
by Lemma \ref{lem:future}.

Write, for $x>0$,
\begin{multline*}
\P\bigg\{\sum_{k=1}^{\tau_1}\big(\Z_k+\W_k^\downarrow\big)>x,\,
\sigma\ge \tau_1\bigg\}\leq \P\bigg\{\sum_{k=1}^{\tau_1-1}
\big(\Z_k+\W_k^\downarrow\big)>x/2,\, \sigma\ge
\tau_1\bigg\}+\P\bigg\{\W_{\tau_1}^\downarrow>x/2,\,
\sigma=\tau_1\bigg\}\\ \leq \P\{AS_{\tau_1}>x/2\}+
\P\big\{\Z_{\tau_1-1} \le A,\, \W^\downarrow_{\tau_1}>x/2\big\}
\end{multline*}
and observe that, in view of \eqref{eq:mom1}, the first summand on
the right-hand side is $o(x^{-3\alpha/2})$ as $x\to\infty$. To
estimate the second term we use a decomposition
$$\W^\downarrow_{\tau_1} = \sum_{i=1}^{\Z_{\tau_1-1}} V_i\quad\text{a.s.},$$
where, for $1\leq i\leq \Z_{\tau_1-1}$, $V_i$ is the number of
progeny in the generations $S_{\tau_1-1}+1,\ldots, S_{\tau_1}-1$
of the $i$th particle in the generation $S_{\tau_1-1}$. We claim
that
\begin{equation}\label{eq:mom2}
\E V_1^\alpha<\infty.
\end{equation}
For the proof, note that $V_1\od \Yc(1,\xi_{\tau_1}-1)$, where
$\xi_{\tau_1}$ is assumed independent of $(\Yc(1,n))_{n\in\N}$.
Consequently, we obtain with the help of Jensen's inequality and
the inequality $\E [\Yc(1,n)]^2 \leq 3n^3$ for $n\in\N$ which is a
consequence of \eqref{eq:harris}
\begin{multline*}
\E V_1^\a=\E [\Yc(1,\xi_{\tau_1}-1)]^\alpha= \sum_{k\geq 0}\E
[\Yc(1,k)]^\alpha \P\{\xi_{\tau_1}-1=k\}\\ \leq \sum_{k\geq 0}
\left(\E [\Yc(1,k)]^2\right)^{\alpha/2} \P\{\xi_{\tau_1}-1=k\}
\leq 3\sum_{k\geq 0} k^{3\alpha/2}\P\{\xi_{\tau_1}-1=k\}\\ =3\E
[\xi_{\tau_1}-1]^{3\alpha/2}\leq 3\E
[S_{\tau_1}]^{3\alpha/2}<\infty,
\end{multline*}
where the last inequality is secured by \eqref{eq:mom1}.

With \eqref{eq:mom2} at hand, we immediately conclude that
\begin{equation*}
\P\big\{ \Z_{\tau_1-1} \le A,\, \W^\downarrow_{\tau_1} > x/2\big\}
\le \P\Big\{\sum_{i=1}^{[A]} V_i > x/2 \Big\}=o(x^{-\alpha}),\quad
x\to\infty
\end{equation*}
because $V_1$, $V_2,\ldots$ are identically distributed. The proof
of Lemma \ref{lem:p3} is complete.
\end{proof}

Before formulating another auxiliary result we recall from Section
\ref{subsec:notation} the notation $Y_1=\sum_{i\geq 1}Z(1,i)$,
where $Z(1,i)$ is the number of progeny residing in the $i$th
generation of the first immigrant, so that $Y_1$ is the total
progeny of the first immigrant.
\begin{lem}\label{lem:new}
Suppose that the assumptions of Lemma \ref{prop:main2} hold. Let
$(Y^\ast_j)_{j\in\N}$ be a sequence of $\P_\omega$-independent
copies of $Y_1$. Then there exists a constant $C>0$ such that
$$\P\bigg\{ \sum_{j=1}^N Y^\ast_j > x \bigg\} \le C N^\alpha
x^{-\alpha},\quad N\in\N.$$
\end{lem}
\begin{proof}
For $k\in\N$, put
\begin{equation}\label{eq:perp}
\wt R_k = \xi_k +\rho_k \xi_{k+1}+\rho_k\rho_{k+1}\xi_{k+2} +
\ldots.
\end{equation}
Recall from Section \ref{sec:rde} that the so defined random
variable is called perpetuity. The Kesten-Grincevi\v{c}ius-Goldie
theorem says that if $(\rho 1)$ holds and $\E \xi^\alpha<\infty$,
then, for all $k\in\N$,
\begin{equation}\label{lem:kes}
\P\{\wt R_k > x\} \sim C x^{-\alpha}, \quad x\to\infty
\end{equation}
for some positive constant $C$ which does not depend on $k$.

Put $Z(1,0):=1$. For $i\in\N_0$, denote by $Z_1(1,i),
Z_2(1,i),\ldots$ $\P_\omega$-independent copies of $Z(1,i)$. Our
proof will be based on the following decomposition which holds
a.s.
\begin{equation*}
\sum_{j=1}^N Y^\ast_j = \sum_{j=1}^N \sum_{i\geq 1} Z_j (1,i) =
\sum_{j=1}^N \sum_{k\geq 1} \xi_k  Z_j(1, S_{k-1}) +\sum_{j=1}^N
\sum_{k\geq 1} \sum_{i=S_{k-1}}^{S_k-1} \big( Z_j(1,i) -
Z_j(1,S_{k-1})\big)=:\mathbb{U}_1 + \mathbb{U}_2.
\end{equation*}
Formula \eqref{eq:perp} implies that, for $k\in\N$, $\xi_k=\wt R_k
- \rho_k\wt R_{k+1}$, whence
\begin{align*}
\mathbb{U}_1 = \sum_{j=1}^N \sum_{k\geq 1} \xi_k  Z_j(1,S_{k-1}) &= \sum_{j=1}^N \sum_{k\geq 1} Z_j(1,S_{k-1})(\wt R_k - \rho_k\wt R_{k+1})\\
& = \sum_{k\geq 1} \bigg(\sum_{j=1}^N \big( Z_j(1,S_{k})- \rho_k
Z_j(1,S_{k-1})\big)\bigg)\wt R_{k+1}+  N \wt R_1.
\end{align*}
Since
\begin{equation}\label{eq:sum}
\sum_{k\geq 1}2^{-1}k^{-2}=\pi^2/12<1,
\end{equation}
and $\wt R_{k+1}$ and $(Z_j(1,S_k), Z_j(1,S_{k-1}),\rho_k)$ are
independent for each $j\in\N$ we obtain with the help of
\eqref{lem:kes}, for $x>0$,
\begin{align*}
&\P\big\{\mathbb{U}_1 >x \big\} \\&\le \sum_{k\geq 1} \P
\bigg\{\bigg|\sum_{j=1}^N \big(Z_j(1,S_{k})-\rho_k
Z_j(1,S_{k-1})\big)\bigg|
\wt R_{k+1} > x/(4k^2)\bigg\}+ \P\big\{N \wt R_1 > x/2 \big\}\\
&\le \sum_{k\geq 1}\int_{[0,\infty)}\P \bigg\{\bigg|\sum_{j=1}^N \big( Z_j(1,S_{k}) - \rho_k Z_j(1,S_{k-1})\big) \bigg| \in {\rm d}s \bigg\}
\P\big\{\wt R_{k+1}>x/(4sk^2)\big\}+ \P\big\{ N\wt R_1 > x/2  \big\}\\
&\le {\rm const}\cdot x^{-\a}\bigg(\sum_{k\geq 1} k^{2\a} \E
\bigg|\sum_{j=1}^N  \big( Z_j(1,S_{k}) - \rho_k
Z_j(1,S_{k-1})\big) \bigg|^\a+N^\alpha\bigg).
\end{align*}
Here and hereafter, ${\rm const}$ denote constants which may be
different on different appearances. To estimate the last term
observe that the equality
\begin{equation}\label{eq:mom}
\E_\omega Z(1,S_i)=\rho_1\cdots \rho_i,\quad i\in\N
\end{equation}
implies that, under $\P_\omega$, $\sum_{j=1}^N \big(Z_j(1,S_k) -
\rho_k Z_j(1,S_{k-1})\big)$ is the sum of iid centered random
variables. With this at hand an application of conditional
Jensen's inequality yields, for $k\in\N$,
\begin{align*}
\E \bigg|  \sum_{j=1}^{N}  \big( Z_j(1,S_{k}) - \rho_k
Z_j(1,S_{k-1}) \big) \bigg|^\a & \le \E \bigg[ \E_\omega
\bigg(\sum_{j=1}^{N}  \big( Z_j(1,S_{k}) - \rho_k Z_j(1,S_{k-1})
\big)\bigg)^2\bigg] ^{\a/2}\\
& = N^{\a/2} \cdot \E \bigg(\E_\omega  \big( Z(1,S_{k}) - \rho_k
Z(1,S_{k-1}) \big)^2   \bigg)^{\a/2}.
\end{align*}
For $k\in\N$ and $1\leq i\leq Z(1, S_{k-1})$, take the $i$th
particle among the progeny in the generation $S_{k-1}$ of the
first immigrant and denote by $V_i^{(k)}$ the number of progeny
residing in the generation $S_k$ of the chosen particle. Then
$$Z(1,S_k) =\sum_{i=1}^{Z(1,S_{k-1})}V_i^{(k)},\quad k\in\N\quad \text{a.s.},$$ and, under $\P_\omega$,
$V_1^{(k)}$, $V_2^{(k)},\ldots$ are independent copies of
$Z(S_{k-1},S_k)$ which are also independent of $Z(1,S_{k-1})$.
Hence,
\begin{equation*}
\E_\omega\Big[\big(Z(1,S_k)-\rho_k Z(1,S_{k-1})\big)^2 \big|
Z(1,S_{k-1})\Big] = Z(1,S_{k-1}) {\rm Var}_\omega
(V_1^{(k)}),\quad k\in\N.
\end{equation*}
Observe that, under $\P_\omega$,
$$V_1^{(k)}\od \sum_{m=1}^{\Zc(S_{k-1},S_k-1)} U^{(k)}_m,\quad k\in\N,$$ where $U_1^{(k)}$, $U_2^{(k)},\ldots$ are $\P_\omega$-independent
random variables with ${\rm Geom}(\lambda_k)$ distribution, and
$\omega$ is assumed independent of $(\Zc(i,j))_{j\geq i\geq 1}$.
This in combination with $\Zc(i,j)\od \Zc(1,j-i+1)$ for fixed
$j\geq i\geq 1$ and \eqref{eq:harris} gives, for $k\in\N$,
\begin{multline*}
{\rm Var}_\omega (V_1^{(k)})=\E_\omega \Zc(S_{k-1},S_k-1){\rm
Var}_{\omega}(U_1^{(k)})+(\E_\omega U_1^{(k)})^2 {\rm Var}_\omega
\Zc(S_{k-1},S_k-1)\\=(\rho_k+\rho_k^2)+2\rho_k^2 (\xi_k-1).
\end{multline*}
Equality \eqref{eq:mom} together with the last formula and
subadditivity of $x\mapsto x^{\alpha/2}$ on $[0,\infty)$ enables
us to conclude that
\begin{align*}
\P\big\{\mathbb{U}_1> x\} &\le \frac{{\rm const}}{x^{\a}} \bigg(\sum_{k\geq 1} k^{2\a} N^{\alpha/2}
\E\Big[ \big(\E_\omega Z(1,S_{k-1}) \big)^{\a/2} \big(\rho_k^{\a/2} + \rho_k^\a + \rho_k^{\alpha} 2^{\alpha}(\xi_k-1)^{\alpha/2}\big) \Big]
+N^\alpha\bigg)\\
&\le \frac{{\rm const}}{x^{\a}}\bigg(\sum_{k\geq 1} k^{2\a}
N^{\alpha/2}(\E \rho^{\alpha/2})^{k-1}+  N^{\alpha} \bigg)={\rm
const} \cdot N^\alpha x^{-\a}.
\end{align*}
To obtain the last inequality we have utilized $\E(\rho^\alpha
\xi^{\alpha/2})<\infty$ which is secured by the assumption
$\E(\rho\xi)^\alpha<\infty$ and the inequality $\E
\rho^{\alpha/2}<1$ which is a consequence of $(\rho 1)$.

To estimate $\mathbb{U}_2$ we proceed similarly but use
additionally Markov's inequality
\begin{align*}
\P\{\mathbb{U}_2 > x\} &= \P\bigg\{\sum_{j=1}^N \sum_{k\geq 1} \bigg(\sum_{i=S_{k-1}}^{S_k-1}(Z_j(1,i)-Z_j(1,S_{k-1}))\bigg) >x \bigg\}\\
&= \sum_{k\geq 1}\P\bigg\{\bigg|\sum_{j=1}^N\bigg(\sum_{i=S_{k-1}}^{S_k-1}(Z_j(1,i)-Z_j(1,S_{k-1}))\bigg)\bigg|>x/(2k^2)\bigg\}\\
&\leq {\rm const}\cdot x^{-\a} \sum_{k\geq 1}k^{2\a}
\E\bigg|\sum_{j=1}^N \bigg( \sum_{i=S_{k-1}}^{S_k-1} (Z_j(1,i) -
Z_j(1,S_{k-1}))\bigg) \bigg|^\a\\ &\leq {\rm const}\cdot
x^{-\a}\sum_{k\geq 1} k^{2\a}\E\bigg(\E_\omega \bigg(\sum_{j=1}^N
\sum_{i=S_{k-1}}^{S_k-1}(Z_j(1,i) - Z_j(1,S_{k-1}))
\bigg)^2\bigg)^{\a/2},\quad x>0.
\end{align*}

For $k\in\N$ and $1\leq r\leq Z(1, S_{k-1})$, take the $r$th
particle among the progeny in the generation $S_{k-1}$ of the
first immigrant and denote by $W_r^{(k)}$ the number of progeny
residing in the generations $S_{k-1},\ldots, S_k-1$ of the chosen
particle. Then
$$\sum_{i=S_{k-1}}^{S_k-1}(Z(1,i)-Z(1,S_{k-1}))=\sum_{r=1}^{Z(1,S_{k-1})}(W_r^{(k)}-(\xi_k-1)),\quad
k\in\N\quad\text{a.s.}$$ Furthermore, under $\P_\omega$,
$W_1^{(k)}$, $W_2^{(k)},\ldots$ are independent random variables
which are independent of $Z(1, S_{k-1})$ and have the same
distribution as $\Yc(1,\xi_k-1)$. Here, as usual, $\omega$ is
assumed independent of $(\Yc(1,n))_{n\in\N}$. Invoking
\eqref{eq:harris} we infer ${\rm Var}_\omega\,(W_r^{(k)})\leq
2\xi_k^3$ and further
\begin{align*}
\P\{\mathbb{U}_2>x\}&\le{\rm const}\cdot x^{-\a}\sum_{k\geq 1}
k^{2\a}N^{\alpha/2}\E\big[\big(\E_\omega Z(1,S_{k-1})
{\rm Var}_\omega (W_1^{(k)})\big)^{\alpha/2}\big] \\
&\le {\rm const}\cdot x^{-\a} \sum_{k\geq 1} k^{2\a} N^{\alpha/2}
(\E \rho^{\alpha/2})^k \E \xi^{3\a/2} \le {\rm const}\cdot
N^{\alpha/2} x^{-\a},\quad x>0.
\end{align*}
The proof of Lemma \ref{lem:new} is complete.
\end{proof}

\begin{proof}[Proof of Lemma \ref{prop:main2}]
Lemma \ref{lem:p3} implies that the contribution of particles
residing in the generations $1,2,\ldots, S_\sigma-1$ is negligible
in the sense that
\begin{equation}\label{eq:pt1}
\P\bigg\{\sum_{k=1}^{\tau_1}\big( \Z_k +  \W_k^\downarrow \big)
>x \bigg\} = \P\bigg\{\mathbb{S}_\sigma + \sum_{i=\sigma+1}^{\tau_1}
\Y^{\downarrow}_i>x,\, \sigma < \tau_1\bigg\} +
o(x^{-\alpha}),\quad x\to\infty.
\end{equation}
Next we prove that
\begin{equation} \label{eq:s3}
\lim_{A\to\infty}\limsup_{x\to\infty}x^\alpha
\P\bigg\{\sum_{i=\sigma(A)+1}^{\tau_1}\Y^{\downarrow}_i>x,\,\sigma(A)<\tau_1\bigg\}=0.
\end{equation}
This means that the contribution of the total progeny of
immigrants arriving in the generations $S_{\sigma(A)}$,
$S_{\sigma(A)}+1,\ldots$ is negligible whenever $A$ is
sufficiently large.

The random variables $\Y^\downarrow_1$, $\Y^\downarrow_2,\ldots$
are identically distributed and, for each $i\in\N$, the random
variables ${\bf 1}_{\{\sigma<i\leq \tau_1\}}={\bf
1}_{\{\sigma<i\}}\cdot(1-{\bf 1}_{\{\tau_1<i\}})$ and
$\Y^{\downarrow}_i$ are independent. Therefore,
\begin{eqnarray}\label{eq:cont}
\P\bigg\{\sum_{i=\sigma(A)+1}^{\tau_1}\Y^\downarrow_i>x,\,\sigma(A)<\tau_1\bigg\}
&\le& \sum_{i\geq 1} \P\big\{{\bf 1}_{\{\s(A) < i \le
\tau_1\}}\Y^\downarrow_i> x/(2 i^2)\big\}\notag\\&=&\sum_{i\geq 1}
\P\{\s(A)<i\le\tau_1\}\P\big\{\Y^\downarrow_1>x/(2 i^2) \big\}
\end{eqnarray}
having utilized \eqref{eq:sum}. Further, observe that
$\Y^\downarrow_1$ is the sum of $\Z_1$ $\P_\omega$-independent
copies of $Y_1=Y(1,\infty)$ which are also $\P$-independent of
$\Z_1$. Hence, using Lemma \ref{lem:new} yields
$$\P\{ \Y_1^\downarrow>x\}\leq C \E \Z_1^\alpha x^{-\alpha},\quad
x>0$$ for some positive constant $C$. The assumptions
$\E\xi^{3\alpha/2}<\infty$ and $\E(\rho\xi)^\alpha<\infty$
guarantee $\E \Z_1^\alpha<\infty$ by Lemma \ref{lem:zmom}.
Continuing \eqref{eq:cont} we obtain
\begin{multline*}
\P\bigg\{\sum_{i=\sigma(A)+1}^{\tau_1}\Y^\downarrow_i>x,\,\sigma(A)<\tau_1\bigg\}
\leq C\E \Z_1^\alpha x^{-\a}\sum_{i\geq 1} i^{2\a} \P\{\sigma(A)<
i \le \tau_1\}\\ \leq C_1\E \Z_1^\alpha x^{-\a}\E \tau_1^{2\a+1}
{\bf 1}_{\{\sigma(A)< \tau_1\}}
\end{multline*}
for a positive constant $C_1$, and \eqref{eq:s3} follows on
letting $A\to\infty$ and recalling that $\E \tau_1^{2\a+1}<\infty$
by Lemma \ref{lem:nu}.

Summarizing it remains to show that
$$\P\{\mathbb{S}_{\sigma(A)}>x,\,\sigma(A)<\tau_1\}~\sim~
C_2(\alpha)x^{-\alpha},\quad x\to\infty,$$ where $C_2(\alpha)$
does not depend on $A$. This can be accomplished by comparing
$\mathbb{S}_{\sigma(A)}$ on the event $\{\sigma(A)<\tau_1\}$ with
$\Z_{\sigma(A)} \widetilde R_{\sigma(A)+1}$  along the lines of
Lemmas 4 and 6 in \cite{kesten1975limit}. We omit the details.
\end{proof}

\subsubsection{Proof of Lemma \ref{lem:1}}

\begin{proof}[Proof of Lemma \ref{lem:1}]
Recall that $$\bar \W_{\tau_1}=W_{S_{\tau_1}} =
\sum_{k=1}^{\tau_1} \W^0_k +  \sum_{k=1}^{\tau_1}(\Z_k +
\W^\downarrow_k)\quad\text{a.s.}$$ According to Lemma
\ref{prop:main2}, $$\P\bigg\{ \sum_{k=1}^{\tau_1} \big( \Z_k +
\W_k^\downarrow  \big)>x \bigg\} \sim C_2(\alpha) x^{-\a},\quad
x\to\infty.$$ By the same reasoning as in the proof of Proposition
\ref{prop:tail_main} (part ${\rm (C1)}$), Lemma \ref{lem:lu1} in
combination with Lemma \ref{lem:nu} and Corollary 3 in
\cite{denisov:foss:korshunov} entails $$
\P\bigg\{\sum_{k=1}^{\tau_1} \W^0_k > x\bigg\} \sim (\E\tau_1)(\E
\vartheta^{\alpha})C_{\ell}x^{-\alpha},\quad x\to\infty.$$ Thus to
prove the lemma it suffices to check that
\begin{equation}\label{eq:pt4}
\P\bigg\{\sum_{k=1}^{\tau_1} \W^0_k > x, \sum_{k=1}^{\tau_1} \big(
\Z_k + \W_k^\downarrow   \big) > x\bigg\} = o( x^{-\alpha}),\quad
x\to\infty,
\end{equation}
see, for example, Lemma B.6.1 in \cite{buraczewski2016stochastic}.

For the proof of \eqref{eq:pt4} we need a number of auxiliary
limit relations. First, according to Lemma \ref{lem:nu} there
exists a constant $C_1>0$ such that
\begin{equation}\label{eq:pt11}
\P\{ \tau_1> C_1 \log x\} = o(x^{-\alpha}),\quad x\to\infty.
\end{equation}
Further, we claim that for any $\delta\in (0,1)$ and large enough
$x$ the following inequalities hold uniformly in $k\in\N$
\begin{equation}\label{eq:pt12}
\P\big\{\W^0_k >x/(C_1\log x),\, \xi_k^2 \le x^{1-\delta}
\big\}\le {\rm const} \cdot x^{-(\alpha +\varepsilon_1)};
\end{equation}
\begin{equation}\label{eq:pt14}
\P\bigg\{\xi_k^2> x^{1-\delta},\, \sum_{j=1}^{(k-1)\wedge \tau_1}
\big( \Z_j + \W_j^\downarrow \big)>x/2\bigg\}\le {\rm const}\cdot
x^{-(\alpha + \varepsilon_1)};
\end{equation}
\begin{equation}\label{eq:pt13}
\P\big\{\xi_k^2> x^{1-\delta},\, \Z_{k-1}> x^{2\delta}\big\}\le
{\rm const}\cdot x^{-(\alpha + \varepsilon_1)},
\end{equation}
where $u\wedge v:=\min(u,v)$ and
$\varepsilon_1:=(\alpha(1-\delta))\wedge (\alpha\delta/2)>0$.

\noindent {\sc Proof of \eqref{eq:pt12}}. Fix any $s>0$ that
satisfies $\delta s>\alpha+\varepsilon_1$. Recall that, under
$\P_\omega$, $\W^0_k\od \Wc_{\xi_k-1}$, where $\omega$ is assumed
independent of $(\Wc_n)_{n\in\N_0}$. This in combination with
Markov's inequality yields
\begin{align*}
\P\big\{\W^0_k >x/(C_1\log x),\, \xi_k^2 \le x^{1-\delta} \big\}&=
\P\big\{\Wc_{\xi_k-1} > x/(C_1\log x),\, \xi_k^2 \le x^{1-\delta}
\big\}\\ & \leq \P\big\{\Wc_{[x^{(1-\delta)/2}]}> x/(C_1\log
x)\big\}\\ & \leq \frac{\E
(\Wc_{[x^{(1-\delta)/2}]})^s}{[x^{(1-\delta)/2}]^{2s}}\frac{(C_1\log
x)^s}{x^{\delta s}} \le {\rm const}  \cdot x^{-(\alpha
+\varepsilon_1)}
\end{align*}
having utilized boundedness of $\E (n^{-2}\Wc_n)^s$ for $n\in\N$,
see Lemma \ref{lem:exp_moments}.

\noindent {\sc Proof of \eqref{eq:pt14}}. For fixed $k\in\N$,
$\xi_k$ is independent of $\sum_{j=1}^{(k-1)\wedge \tau_1}
\big(\Z_j + \W_j^\downarrow \big)$. Using this, Lemma
\ref{prop:main2} and the assumptions of Lemma \ref{lem:1} we
conclude that
\begin{align*}
\P\bigg\{\xi_k^2> x^{1-\delta},\, \sum_{j=1}^{(k-1)\wedge \tau_1}
\big(\Z_j + \W_j^\downarrow \big)>x/2\bigg\} &\le \P\{\xi^2>
x^{1-\delta}\} \P\bigg \{\sum_{j=1}^{\tau_1}\big(\Z_j +
\W_j^\downarrow \big)>x/2\bigg\}\\&\sim~ 2^\alpha C_\ell
C_2(\alpha)x^{-\alpha}x^{-\alpha(1-\delta)}\le {\rm const}\cdot
x^{-(\alpha +\varepsilon_1)}.
\end{align*}

\noindent {\sc Proof of \eqref{eq:pt13}}. Observing that, for
every fixed $k\in\N$, $\xi_k$ is independent of $\Z_{k-1}$ and
invoking Lemma \ref{lem:zmom} with $\kappa =3\alpha/4$ we obtain
\begin{align*} \P\big\{\xi_k^2> x^{1-\delta},\,
\Z_{k-1}>x^{2\delta}\big\} &= \P\big\{\xi_k^2> x^{1-\delta}
\big\}\P\big\{ \Z_{k-1}>
x^{2\delta}\big\} \\
&\le {\rm const}\cdot C C_\ell x^{-\alpha(1-\delta)}
x^{-(3/2)\alpha\delta} \le {\rm const}  \cdot  x^{-(\alpha
+\varepsilon_1)}.
\end{align*}

Combining \eqref{eq:pt11}, \eqref{eq:pt12}, \eqref{eq:pt14} and
\eqref{eq:pt13} yields, for any $\delta\in(0,1)$,
\begin{align*}
\P\bigg\{\sum_{k=1}^{\tau_1}& \W^0_k > x,\,\sum_{j=1}^{\tau_1} \big( \Z_j + \W_j^\downarrow   \big) > x\bigg\}\\
&\overset{\eqref{eq:pt11}}{\le} \P\bigg\{\sum_{k=1}^{\tau_1}
\W^0_k > x,\,
\sum_{j=1}^{\tau_1} \big( \Z_j + \W_j^\downarrow   \big) > x,\, \tau_1\leq C_1 \log x\bigg\} + o(x^{-\alpha})\\
& \le \sum_{k\leq C_1 \log x} \P\bigg\{\W^0_k > \frac{x}{C_1\log
x},\,
\sum_{j=1}^{\tau_1} \big( \Z_j + \W_j^\downarrow   \big) > x,\, \tau_1\leq C_1 \log x\bigg\} + o(x^{-\alpha})\\
& \overset{\eqref{eq:pt12}}{\le} \sum_{k\leq C_1 \log x} \P\bigg\{
\xi_k^2 > x^{1-\delta},\,
\sum_{j=1}^{\tau_1} \big( \Z_j + \W_j^\downarrow\big) > x,\, \tau_1\leq C_1 \log x\bigg\} + o(x^{-\alpha})\\
& \overset{\eqref{eq:pt14}}{\le} \sum_{k\leq C_1 \log x}
\P\bigg\{\xi_k^2 > x^{1-\delta},\,
\sum_{j=k}^{\tau_1} \big( \Z_j + \W_j^\downarrow\big) > x/2,\, k \le \tau_1,\, \tau_1\leq C_1 \log x\bigg\} + o(x^{-\alpha})\\
&\overset{\eqref{eq:pt13}}{\le} \sum_{k\leq C_1 \log x} \P\bigg\{
\xi_k^2 > x^{1-\delta},\, \sum_{j=k}^{\tau_1} \big( \Z_j +
\W_j^\downarrow   \big) > x/2,\, k \le \tau_1,\, \Z_{k-1}\leq
x^{2\delta} \bigg\} + o(x^{-\alpha}).
\end{align*}
Now \eqref{eq:pt4} follows if we can show that for some $\delta\in
(0,1)$ the following inequality holds uniformly in $k$
\begin{equation}\label{eq:pt16}
\P\big\{ \xi_k^2 > x^{1-\delta},\, \Z_k + \W_k^\downarrow > x/4,\,
\Z_{k-1} \le x^{2\delta} \big\} \le {\rm const}\cdot
x^{-(\alpha+\varepsilon_2)}
\end{equation}
for large enough $x$ and some $\varepsilon_2>0$ to be specified
below, and that
\begin{equation}\label{eq:pt17}
\sum_{k\leq C_1\log x} \P\bigg\{ \xi_k^2 > x^{1-\delta},\,
\sum_{j=k+1}^{\tau_1} \big( \Z_j + \W_j^\downarrow   \big) >
x/4\bigg\}=o(x^{-\alpha}),\quad x\to\infty.
\end{equation}

\noindent {\sc Proof of \eqref{eq:pt16}}. Observe that $$ \Z_k +
\W_k^\downarrow = \sum_{i=1}^{\Z_{k-1}}V^{(k)}_i\quad
\text{a.s.},$$ where, for $k\in\N$ and $1\leq i\leq \Z_{k-1}$,
$V^{(k)}_i$ denotes the number of progeny residing in the
generations $S_{k-1}+1$ through $S_k$ of the $i$th particle in the
generation $S_{k-1}$. Clearly, for fixed $k\in\N$,
$V^{(k)}_1,\ldots, V^{(k)}_{\Z_{k-1}}$ are independent of
$\Z_{k-1}$ and have the same distribution as $$\Yc(1,\xi_k-1)
+\sum_{j=1}^{\Zc(1,\xi_k-1)}U^{(k)}_j,$$ where
$(\Yc(1,n))_{n\in\N}$ and $(\Zc(1,n))_{n\in\N}$ are assumed
independent of $(\xi_k, \rho_k)$, $U^{(k)}_1, U^{(k)}_2,\ldots$
have ${\rm Geom}(\lambda_k)$ distribution and, given
$(\xi_k,\rho_k)$, they are independent of $\Zc(1,\xi_k-1)$. In
particular, $\E\big(V^{(k)}_1|(\xi_k,\rho_k)\big)=\xi_k-1+\rho_k$
in view of \eqref{eq:harris}. With this at hand we obtain
\begin{align*}
&\P\big\{ \xi_k^2 > x^{1-\delta}, \Z_k + \W_k^\downarrow > x/4,
\Z_{k-1} \le x^{2\delta} \big\}\\&=\E{\bf 1}_{\{\Z_{k-1}\leq
x^{2\delta}\}}\P\bigg\{\xi_k^2
> x^{1-\delta}, \sum_{i=1}^{\Z_{k-1}}V^{(k)}_i >
x/4\bigg|\Z_{k-1}\bigg\}\\&\le \E\Z_{k-1}{\bf 1}_{\{\Z_{k-1}\leq
x^{2\delta}\}}\P\bigg\{\xi_k^2 > x^{1-\delta}, V^{(k)}_1 >
x/(4\Z_{k-1})\bigg|\Z_{k-1}\bigg\}\\&\leq x^{2\delta} \P\big\{\xi_k^2 > x^{1-\delta}, V^{(k)}_1> x^{1-2\delta}/4 \bigg\}\\
&\leq {\rm const}\cdot x^{2\delta} \E\bigg[{\bf 1}_{\{\xi_k^2 >
x^{1-\delta}\}}
\E\bigg[\frac{(V^{(k)}_1)^r}{x^{r(1-2\delta)}}\bigg|(\xi_k,\rho_k)\bigg]\bigg]\\&\leq
{\rm const}\cdot x^{2\delta-r(1-2\delta)} \E\Big[{\bf
1}_{\{\xi_k^2 > x^{1-\delta}\}}
\big( \E[ V^{(k)}_1 | (\xi_k,\rho_k)] \big)^r\Big]\\
&\le {\rm const}\cdot x^{2\delta-r(1-2\delta)} \E\Big[ {\bf
1}_{\{\xi_k^2 > x^{1-\delta}\}}  (\xi_k + \rho_k)^r\Big]
\end{align*}
for $k\in\N$, large enough $x$ and any $r\in (0,1]$, having
utilized conditional Jensen's inequality for the penultimate step.
By assumption $\E\rho^\gamma<\infty$ and $\E\xi^\gamma<\infty$ for
some $\gamma\in (\alpha, 2\alpha)$. Taking $r\in (0,\gamma)$ and
applying H\"older's inequality with parameters $\gamma/(\gamma-r)$
and $\gamma/r$ we arrive at
\begin{align*}
\P\big\{ \xi_k^2 > x^{1-\delta}, \Z_k + \W_k^\downarrow > x/4,
\Z_{k-1} &< x^{2\delta}  \big\}\\&\le {\rm const} \cdot \big( \E
\xi_k^{\gamma} + \E \rho_k^{\gamma} \big)^{r/\gamma}
x^{2\delta-r(1-2\delta)-(1-\delta)\alpha(1-r/\gamma)}.
\end{align*}
Pick any $\rho\in (0, (1-\alpha/\gamma)/(2+\alpha))$ and then any
$r\in (0,\gamma\wedge
((1-\alpha/\gamma-\rho(2+\alpha))/(\rho(2-\alpha/\gamma))))$.
Setting now $\delta=\rho r$ (so that $\delta\in (0,1)$) we obtain
\eqref{eq:pt16} with
$\varepsilon_2:=-\alpha-2\delta+r(1-2\delta)+(1-\delta)\alpha(1-r/\gamma)$.
Throughout the rest of the proof $\delta$ always denotes the
number chosen above.

\noindent {\sc Proof of \eqref{eq:pt17}}. For $k\in\N$ and $1\leq
i\leq \Z_k$, denote by $Y^{(k)}_i$ the total progeny of the $i$th
particle in the generation $S_k$. Further, for $k\in\N$ and $j\geq
k+2$, denote by $\W_j^\downarrow(k)$ the number of progeny in the
generations $S_{j-1}, S_{j-1}+1,\ldots, S_j-1$ of the immigrants
arriving in the generations $S_k, S_k+1,\ldots, S_{j-1}-1$. Then
$$\sum_{j=k+1}^{\tau_1} \big( \Z_j + \W_j^\downarrow   \big) =
\sum_{i=1}^{\Z_k}Y^{(k)}_i
+\sum_{j=k+2}^{\tau_1}\W_j^\downarrow(k)\quad \text{a.s.}$$ and
thereupon, for $x>0$,
\begin{align*}
\P\bigg\{ \xi_k^2 > x^{1-\delta},\, \sum_{j=k+1}^{\tau_1}
\big(\Z_j + \W_j^\downarrow   \big) > x/4\bigg\}&\leq
\P\bigg\{\xi_k^2 > x^{1-\delta},\, \sum_{i=1}^{\Z_k} Y^{(k)}_i>
x/8\bigg\}\\&+\P\bigg\{ \xi_k^2> x^{1-\delta},\,
\sum_{j=k+2}^{\tau_1}\W_j^\downarrow(k)>x/8\bigg\}\\&=:I_1(x)+I_2(x).
\end{align*}
Since, for fixed $k\in\N$,
$\sum_{i=k+2}^{\tau_1}\W_i^\downarrow(k)$ is independent of
$\xi_k$ we obtain with the help of a crude estimate
$$\sum_{i=k+2}^{\tau_1}\W_i^\downarrow(k)\leq \sum_{i=1}^{\tau_1}\big(\Z_i+\W_i^\downarrow\big),\quad k\in\N\quad \text{a.s.}$$ and Lemma
\ref{prop:main2}
$$I_2(x)\leq \P\big\{ \xi_k^2 > x^{1-\delta} \big\} \P\bigg\{\sum_{i=1}^{\tau_1}\big(\Z_i+\W_i^\downarrow\big)>x/8\bigg\}
\le {\rm const} \cdot x^{-\alpha(1-\delta)} x^{-\alpha}$$ for
large enough $x$. Of course, this entails $\sum_{k\leq C_1\log
x}I_2(x)=o(x^{-\alpha})$ as $x\to\infty$.

To estimate $I_1(x)$ we note that, for fixed $k\in\N$, under
$\P\{\cdot|\omega, \Z_k\}$, $Y^{(k)}_1,\ldots, Y^{(k)}_{\Z_k}$ are
independent copies of $Y(1,\infty)$. Furthermore, these random
variables are $\P$-independent of $\Z_k$ and $\xi_k$. Invoking
Lemma \ref{lem:new} and conditional Jensen's inequality yields
\begin{align*}
\P\bigg\{ \xi_k^2 > x^{1-\delta}, \sum_{i=1}^{\Z_k} Y^{(k)}_i
>x/8\bigg\}&= \E\bigg[ {\bf 1}_{\{ \xi_k^2
> x^{1-\delta}\}}\P\bigg[ \sum_{i=1}^{\Z_k} Y^{(k)}_i >x/8\Big|\xi_k, \Z_k \bigg]\bigg]\\
&\le {\rm const}\cdot x^{-\alpha} \E\big[ {\bf 1}_{\{\xi_k^2 >
x^{1-\delta}\}}\Z_k^\alpha\big]\\&= {\rm const}\cdot x^{-\alpha}
\E\left[ {\bf 1}_{\{\xi_k^2 > x^{1-\delta}\}}\E_\omega
\left[\Z_k^\alpha |\Z_{k-1}\right]\right]\\&\le {\rm const}\cdot
x^{-\alpha} \E\left[ {\bf 1}_{\{\xi_k^2 > x^{1-\delta}\}}
\left(\E_\omega\left[\Z_k^2|\Z_{k-1}\right]\right)^{\alpha/2}\right].
\end{align*}
Inequality \eqref{eq:aux1} was obtained in the proof of Lemma
\ref{lem:zmom} under the assumption $\kappa\in (1,2]$. However, by
the same reasoning it also holds for $\kappa\in (0,2]$. Using
\eqref{eq:aux1} in combination with the fact that $\xi\geq 1$
a.s.\ we infer
$$\left(\E_\omega\left[\Z_k^2|\Z_{k-1}\right]\right)^{\alpha/2} \le
{\rm const}\cdot \big(\Z_{k-1}^\alpha
(\rho_k\xi_k)^\alpha+\Z_{k-1}^{\alpha/2}\big((\rho_k\xi_k)^\alpha+(\rho_k\xi_k)^{\alpha/2}\big)+(\rho_k\xi_k)^\alpha+(\rho_k\xi_k)^{\alpha/2}\big)$$
and thereupon
\begin{align*}
\E\left[{\bf 1}_{\{\xi_k^2 >
x^{1-\delta}\}}\left(\E_\omega\left[\Z_k^2|\Z_{k-1}\right]\right)^{\alpha/2}\right]&
\leq {\rm const}\cdot \big(k \E (\rho\xi)^\alpha{\bf
1}_{\{\xi^2>x^{1-\delta}\}}+\E(\rho \xi)^{\alpha/2}{\bf
1}_{\{\xi^2>x^{1-\delta}\}}\big)\\&\leq {\rm const}\cdot
x^{-\varepsilon(1-\delta)/2}\big(k
\E\rho^\alpha\xi^{\alpha+\varepsilon}+\E\rho^{\alpha/2}\xi^{\alpha/2+\varepsilon}\big)\\&\leq
{\rm const}\cdot kx^{-\varepsilon(1-\delta)/2}
\end{align*}
by Lemma \ref{lem:zmom} and the assumption $\E\rho^\alpha
\xi^{\alpha+\varepsilon}<\infty$ for some $\varepsilon>0$. The
latter entails
$$
\sum_{k\leq C_1\log x}I_1(x)=o(x^{-\alpha}),\quad x\to\infty.
$$
The proof of Lemma \ref{lem:1} is complete.
\end{proof}

\appendix
\section{Appendix}

Lemma \ref{lem:future} is an important ingredient in the proof of
Proposition \ref{prop:tail_main}, part ${\rm (C1)}$. In its
formulation we use the notion of a random variable which does not
depend on the future of a sequence of random variables. The
corresponding definition can be found at the beginning of Section
\ref{sec:tails}.
\begin{lem}\label{lem:future}
Let $(\theta_i)_{i\in\N}$ be a sequence of iid nonnegative random
variables and $T$ a nonnegative integer-valued random variable
which does not depend on the future of the sequence
$(\theta_i)_{i\in\N}$. Assume that $\E\theta_1^s<\infty$ for some
$s>0$ and that $\E e^{\lambda T}<\infty$ for some $\lambda>0$.
Then $\E (\sum_{i=1}^T \theta_i)^s<\infty$.
\end{lem}
\begin{proof}
Set $R_0:=0$ and $R_i:=\theta_1+\ldots+\theta_i$ for $i\in\N$. By
assumption, for fixed $i\in\N$, $\theta_i$ is independent of
$(R_{i-1}, {\bf 1}_{\{T\geq i\}})$.

The result is trivial when $s\in (0,1]$. Indeed, we use
subadditivity of $x\mapsto x^s$ on $[0,\infty)$ together with the
aforementioned independence to conclude that $$\E
\Big(\sum_{i=1}^T \theta_i\Big)^s\leq \sum_{i\geq 1}\E \theta_i^s
{\bf 1}_{\{T\geq i\}}=\E\theta_1^s \E T<\infty.$$

Assume now that $s>1$. Invoking the inequality
$$(x+y)^s\leq x^s+sy(x+y)^{s-1},\quad x,y\geq 0$$ which is secured
by the mean value theorem for differentiable functions we obtain
$$R_{T\wedge i}^s\leq R^s_{T\wedge (i-1)}+s\theta_i R_i^{s-1}{\bf
1}_{\{T\geq i\}},\quad i\in\N.$$ Iterating this yields
$$R_{T\wedge n}^s\leq s\sum_{i=1}^n \theta_i R_i^{s-1}{\bf
1}_{\{T\geq i\}},\quad n\in\N.$$ Therefore, it is enough to check
that
$$A:=\E \sum_{i\geq 1}\theta_i R_i^{s-1}{\bf 1}_{\{T\geq
i\}}<\infty.$$ Using once again the aforementioned independence
together with the inequality $$(x+y)^{s-1}\leq C_s
(x^{s-1}+y^{s-1}),\quad x,y\geq 0,$$ where $C_s:=\max(2^{s-2},1)$,
we infer $$A\leq C_s \E \sum_{i\geq 1}\theta_i
(R_{i-1}^{s-1}+\theta_i^{s-1}){\bf 1}_{\{T\geq i\}}=C_s \E
\theta_1\sum_{i\geq 1}\E R_{i-1}^{s-1}{\bf 1}_{\{T\geq i\}}+ C_s
\E \theta_1^s \E T.$$ Left with checking convergence of the series
we appeal to H\"{o}lder's inequality in conjunction with convexity
of $x\mapsto x^s$ on $[0,\infty)$ to get
$$\E R_{i-1}^{s-1}{\bf 1}_{\{T\geq i\}}\leq [\E R_{i-1}^s]^{(s-1)/s}[\P\{T\geq i\}]^{1/s}\leq i^{s-1}[\E\theta_1^s]^{(s-1)/s}[\P\{T\geq i\}]^{1/s}.$$
Since $[\P\{T\geq i\}]^{1/s}$ decreases at least exponentially in
$i$, $\E R_{i-1}^{s-1}{\bf 1}_{\{T\geq i\}}$ is the general term
of converging series. The proof of Lemma \ref{lem:future} is
complete.
\end{proof}

The remaining part of the Appendix is concerned with the proof of
Lemma~\ref{lem:nu}. In essence the lemma follows from the
arguments presented by Key~\cite{key1987limiting} who considered a
model very similar to ours. For $n\in\N$ and $1\leq k\leq n$, set
\begin{equation*}
\Z(k,n) = \sum_{j=S_{k-1}+1}^{S_k} Z(j, S_n)
\end{equation*}
and observe that, under $\P_\omega$, $\Z(1,n),\ldots, \Z(n,n)$ are
independent. The following representation holds
\begin{equation*}
\Z(0)=0,\quad \Z_n = \sum_{k=1}^{n-1}\Z(k,n) + \Z(n,n),\quad
n\in\N
\end{equation*}
which shows that $(\Z_n)_{n\in\N_0}$ is a branching process in a
random environment with the random number $\Z(k,k)$ of immigrants
in the $k$th generation. The basic observation for what follows is
that $(\Z_n)_{n \geq 0}$ has the structure similar to that of the
branching process investigated by Key~\cite{key1987limiting}. The
main difference manifests in the term $\Z(n,n)$ which is absent in
Key's model. It is curious that the branching process in
\cite{key1987limiting} is similar to our $(Z_n)_{n\in\N_0}$ in
that the immigrants arriving in the generation $n$ only affect the
system by their offspring residing in the generation $n+1$. In
particular, neither Key's process nor our $(Z_n)_{n\in\N_0}$
counts immigrants, whereas $(\Z_n)_{n\in\N_0}$ does.

Even though $(\Z_n)_{n \geq 0}$ and Key's process are slightly
different it is natural to expect that sufficient conditions
ensuring finiteness of power and exponential moments of the first
extinction time should be similar. While demonstrating that this
is indeed the case we shall only point out principal changes with
respect to Key's arguments.

Denote by
\begin{equation*}
p(n,k) = \P_{\omega} \{\Z(1,n)=k\: | \: \Z(1,n-1)=1 \}, \quad
n\geq 2,~k\in\N_0
\end{equation*}
and
\begin{equation*}
a(n,k) = \PP_{\omega}\{ \Z(n,n)=k\},\quad n\in\N,~ k\in\N_0
\end{equation*}
the quenched reproduction and immigration distribution in the
generation $n$, respectively. It can be checked that the mean of
the quenched reproduction distribution is
\begin{equation*}
M(n) =\sum_{k\geq 0}k p(n,k) = \E_\omega [\ZZ(1,n) |
\ZZ(1,n-1)=1]=\rho_n,\quad n\geq 2
\end{equation*}
and that the quenched expected number of immigrants is
\begin{equation*}
I(n) =\sum_{k\geq 0}k a(n,k)= \E_\omega [\Z(n,n)] =
\rho_n\xi_n,\quad n\in\N.
\end{equation*}

Lemma \ref{prop:AppProp} is a counterpart of Theorem 3.3
in~\cite{key1987limiting}.
\begin{lem}\label{prop:AppProp}
Assume that $\E\log \rho\in [-\infty, 0)$ and $\E\log^+ \xi
<\infty$. Then, for $k\in\N_0$, $\pi(k) = \lim_{n \to \infty}
\P\{\Z_n=k\}$ exists and defines a probability distribution on
$\N$. If additionally
\begin{equation}\label{eq:AppCond}
\P\{p(2,0)>0, \: a(2,0)>0\}>0,
\end{equation}
then $\pi(0)>0$.
\end{lem}
\begin{proof}[Sketch of proof]
As far as the first claim is concerned, the proofs of Lemmas~2.1,
2.2, 3.1, 3.2 in \cite{key1987limiting} only require inessential
changes concerning the range of summation. The second claim
follows after a minor alteration, namely the term $q(n,k)$
appearing in the proof of Theorem 3.3 in \cite{key1987limiting}
should be changed to
\begin{equation*}
q(n,k)=\P_\omega\{\Z_{n+1}=0 \: | \: \Z_n=k\}= p(n+1,0)^k
a(n+1,0),\quad n\in\N, k\in\N_0.
\end{equation*}
The sequence $(q(1,k))_{k\in\N_0}$ must be positive which
justifies condition \eqref{eq:AppCond}. The corresponding
condition in \cite{key1987limiting} is slightly different.
\end{proof}

We are ready to prove Lemma \ref{lem:nu}.
\begin{proof}[Proof of Lemma \ref{lem:nu}]
The present proof is very similar to that of Theorem~4.2
in~\cite{key1987limiting}. Put
\begin{equation*}
v(n):= \PP\{ \tau_1 >n \},\quad n\in\N_0
\end{equation*}
and
\begin{equation*}
V(x): = \sum_{n\geq 1} v(n)x^n,\quad x\geq 0
\end{equation*}
which may be finite or infinite. While finiteness of $\E\tau_1$ is
equivalent to $V(1)<\infty$, finiteness of some exponential moment
of $\tau_1$ is equivalent to $V(x)<\infty$ for some $x>1$.

For $n\in\N$, put
\begin{equation*}
h(k,n):= \PP\bigg\{ \Z(k,n)>0, \: \sum_{j=k+1}^n \Z(j,n)=0
\bigg\},\quad 1\leq k\leq n
\end{equation*}
(with the usual convention that $h(n,n)=\PP\{\Z(n,n)>0\}$) and
note that $h(k,n) = h(1,n-k+1)$ for $1\leq k\leq n$. Now we use a
decomposition
\begin{align*}
v(n)  =&  \PP \{ \tau_1>n, \: \Z_n>0 \} = \PP\bigg\{ \tau_1>n , \: \sum_{k=1}^n\Z(k,n)>0 \bigg\}\\
= &     \sum_{k=1}^{n-1}\P\bigg\{ \tau_1>n,\: \Z(k,n) >0, \:
\sum_{j=k+1}^n\Z(j,n)=0\bigg\} + \PP\{\tau_1>n,\:  \Z(n,n)>0\}.
\end{align*}
in combination with
\begin{align*}
& \P\bigg\{ \tau_1>n,\: \Z(k,n) >0, \:
\sum_{j=k+1}^n\Z(j,n)=0\bigg\} = \P\bigg\{ \tau_1>k-1,\: \Z(k,n)
>0, \: \sum_{j=k+1}^n\Z(j,n)=0\bigg\} \\& = \P\{ \tau_1>k-1 \}
\P\bigg\{ \Z(k,n) >0, \:
\sum_{j=k+1}^n\Z(j,n)=0\bigg\}=v(k-1)h(k,n) = v(k-1)h(1,n-k+1)
\end{align*}
which holds for $1\leq k\leq n$ to obtain
\begin{equation*}
v(n) = \sum_{k=0}^{n-1} v(k) h(1,n-k),\quad n\in\N.
\end{equation*}
This convolution equation is equivalent to
\begin{equation*}
V(x)= \frac{H(x)}{1-H(x)},\quad x\geq 0
\end{equation*}
(the possibility that both sides are infinite is not excluded),
where
\begin{equation*}
H(x) = \sum_{j\geq 1} h(1,j) x^j,\quad x\geq 0.
\end{equation*}

Now $\E\tau_1<\infty$ follows from
\begin{equation*}
H(1) = \sum_{j\geq 1} h(1,j) = \lim_{n \to \infty}
\PP\{\ZZ_n>0\}=1-\pi(0)
\end{equation*}
once we can show that $\pi(0)>0$. To this end, we recall that
$(Z_n)_{n\in\N_0}$ is governed by a geometric distribution, whence
$$p(n,0)\geq \lambda_n {\bf 1}_{\{\xi_n=1\}}+2^{-1}{\bf 1}_{\{\xi_n>1\}}\geq \lambda_n\wedge 1/2,\quad n\geq 2$$
and $$a(n,0)=\sum_{j\geq 1}\frac{\lambda_n}{j-(j-1)\lambda_n}{\bf
1}_{\{\xi_n=j\}} \geq \lambda_n \sum_{j\geq 1}j^{-1}{\bf
1}_{\{\xi_n=j\}},\quad n\in\N.$$ These inequalities ensure
\eqref{eq:AppCond} and thereupon $\pi(0)>0$ by
Lemma~\ref{prop:AppProp}.

To prove finiteness of some exponential moment pick $\delta\in
(0,1)$ such that
\begin{equation*}
\E(\rho \xi)^\delta<\infty \quad \mbox{and} \quad
r:=\E\rho^\delta<1.
\end{equation*}
Existence of such a $\delta$ is justified by assumptions and the
Cauchy-Schwarz inequality. In view of
\begin{align*}
h(1,j) & \leq \PP \{\ZZ(1,j) \geq 1\}\leq\E (\E_\omega
\Z(1,j))^\delta = \E(\rho \xi)^\delta r^{j-1}
\end{align*}
we infer that the radius of convergence of $H$ is greater than
one. This in combination with $H(1)<1$ implies that $H(x) <1$ and
thereupon $V(x)<\infty$ for some $x>1$.
\end{proof}

\section*{Acknowledgment}

D. Buraczewski and P. Dyszewski were partially supported by the
National Science Center, Poland (Sonata Bis, grant number
DEC-2014/14/E/ST1/00588). A. Marynych was partially supported by
the Return Fellowship of the Alexander von Humboldt Foundation. A
part of this work was done while A. Iksanov and A. Marynych were
visiting Wroclaw in February 2018. They gratefully acknowledge
hospitality and the financial support.

\bibliographystyle{plain}

\vspace{1cm}

\footnotesize

\textsc{Dariusz Buraczewski:} Mathematical Institute, University
of Wroclaw, 50-384 Wroclaw, Poland\\
\textit{E-mail}: \texttt{dbura@math.uni.wroc.pl}

\bigskip

\textsc{Piotr Dyszewski:} Mathematical Institute, University of
Wroclaw, 50-384 Wroclaw, Poland\\
\textit{E-mail}: \texttt{pdysz@math.uni.wroc.pl}

\bigskip

\textsc{Alexander Iksanov:} Faculty of Computer Science and Cybernetics, Taras Shev\-chen\-ko National University of Kyiv, 01601 Kyiv, Ukraine\\
\textit{E-mail}: \texttt{iksan@univ.kiev.ua}

\bigskip

\textsc{Alexander Marynych:} Faculty of Computer Science and Cybernetics, Taras Shev\-chen\-ko National University of Kyiv, 01601 Kyiv, Ukraine\\
\textit{E-mail}: \texttt{marynych@unicyb.kiev.ua}

\bigskip

\textsc{Alexander Roitershtein:} Department of Mathematics, Iowa State University, Ames, IA 50011, USA\\
\textit{E-mail}: \texttt{roiterst@iastate.edu}

\end{document}